\documentclass[11pt,  reqno]{amsart}
 \usepackage[margin=1in]{geometry}
\usepackage[latin1]{inputenc}
\usepackage{amsmath}
\usepackage{amsfonts}
\usepackage{hyperref}
\usepackage{color}
\usepackage[all,cmtip]{xy}

\usepackage{psfrag}      
               
\usepackage{graphics}
\usepackage[dvips]{graphicx}

\usepackage{boxedminipage}
\usepackage{color}

\DeclareMathOperator{\T}{\mathbb{T}}
\DeclareMathOperator{\Z}{\mathbb{Z}}

\DeclareMathOperator{\N}{\mathbb{N}}
\DeclareMathOperator{\TP}{\mathbb{T P}}
\DeclareMathOperator{\R}{\mathbb{R}}
\DeclareMathOperator{\CP}{\mathbb{C P}}

\DeclareMathOperator{\C}{\mathbb{C}}

\DeclareMathOperator{\Oc}{\mathcal{O}}

\DeclareMathOperator{\K}{\mathcal{K}}

\newcommand{\comment}[1]{}

\newtheorem{theorem}{Theorem}[section]
\newtheorem{lemma}[theorem]{Lemma}

\newtheorem{proposition}[theorem]{Proposition}
\newtheorem{corollary}[theorem]{Corollary}
\newtheorem{definition}[theorem]{Definition}
\newtheorem{exa}[theorem]{Example}
\newtheorem{example}[theorem]{Example}

\newenvironment{remark}[1][Remark]{\begin{trivlist}
\item[\hskip \labelsep {\bfseries #1}]}{\end{trivlist}}

\begin{document}

\title{A tropical intersection product in  matroidal fans}
\author{Kristin M. Shaw}

\address{Kristin Shaw, Section de Math\'ematiques, Universit\'e de Gen\`eve, Villa Battelle, 7 Route de Drize, 1227 Carouge, Suisse.}

\email{Kristin.Shaw@unige.ch}

\begin{abstract}We construct an intersection product on tropical cycles contained in   the Bergman fan of a matroid. To do this we first establish a connection between the operations of deletion and restriction in matroid theory and tropical modifications as defined by Mikhalkin in \cite{MikICM}. This product generalises the product of  Allermann  and Rau \cite{AllRau}, and Allermann \cite{Aller} and also provides an alternative procedure for intersecting cycles which is not based on intersecting with Cartier divisors.  Also, we simplify the definition in the case of one dimensional fan cycles in two dimensional matroidal fans and given an application of the   intersection product to realisability questions in tropical geometry.
\end{abstract}

\maketitle

\begin{section}{Introduction}

One of the main goals of tropical geometry is to study classical algebraic geometry via  polyhedral complexes. 
Tropicalisations of subvarieties of $(\C^*)^n$ are  rational polyhedral complexes in $\R^n$ equipped with positive integer weights and satisfying the so-called balancing condition. For this reason tropical subvarities of $\R^n$ are considered to be  polyhedral complexes with this added structure,  \cite{FirstSteps}, \cite{MikICM}. 

Before the advent of tropical geometry, Bergman fans were initially defined to be the \textit{logarithmic limit sets} of complex algebraic varieties \cite{Berg}. When equipped with appropriate weights they are tropical varieties in the above sense.  For varieties defined by linear ideals, Sturmfels showed that the Bergman fan depends only on the underlying  matroid. In addition,  he  generalised the Bergman fan construction to any loopless matroid \cite{SturmPoly}. Following this, an explicit construction of the fan involving matroid polytopes was given in \cite{SturmFeich}, and its relation to the lattice of flats of the corresponding matroid has been studied in \cite{ArdKli}. 

Bergman fans of matroids highlight the fact that not all tropical subvarieties have a classical counter-part. It is well-known that there exist matroids not representable over any field. However, in the tropics, as mentioned above every matroid has a geometric representation as a polyhedral fan.   In  this article,  the Bergman fan of a matroid will also simply be called a matroidal fan. 

Matroidal fans have many nice properties making them candidates for the local models of tropical non-singular spaces. Tropical linear spaces as studied by Speyer  \cite{Speyer} and Speyer and Sturmfels \cite{tropGrass} are  locally matroidal fans.  In addition, any codimension one cycle on a matroidal fan may be expressed as a tropical Cartier divisor, this is proved  in Section \ref{sec:Berg}. A particular case of this was proved by Allermann in \cite{Aller} for so-called ``tropical linear fans". These are skeleta of tropical hyperplanes, and not all matroidal fans arise in this way. 
However we show here that every matroidal fan of dimension can be obtained from $\T^n$ by a sequence of tropical modifications, see Subsection \ref{sec:Berg}. 
 A tropical  modification can be thought of as a re-embedding of a tropical cycle. %by taking the graph of a tropical function in $\R^n$, (this is different from the definition in \cite{MikICM} where the graph is considered in $\T^n$).  
 For this reason they are considered to be models of  tropical affine space.
%\ref{sec:Berg}.
The aim of this paper is to give a procedure for intersecting tropical cycles contained in matroidal fans which can be applied to more general smooth tropical spaces. 

When the ambient space is $\R^n$, a tropical intersection product already exists, and is known as \textit{stable intersection}, see \cite{FirstSteps}, \cite{MikICM}. This intersection product is  related to the fan displacement rule in toric intersection theory. 
In this case there are also various correspondence theorems relating the intersection of classical algebraic varieties in $(\C^*)^n$  to the stable intersection of their tropicalisations,  see \cite{Pay1}, \cite{Br16}. The stable intersection of two tropical  cycles $A, B \subset \R^n$ is supported on the skeleton $(A \cap B)^{(k)}$ where $k$ is the expected dimension of intersection. Moreover,  the weights of the facets of the intersection are determined by the  local structure of the complexes, see \cite{Katz}, \cite{Rau}. 

One of the main differences and advantages  of  tropical stable intersection over  classical theories is that products are defined on the level of cycles, even in the case of self-intersection.  This greatly contrasts the situation in classical algebraic geometry, where some notion of equivalence is necessary in order to define an intersection product. A principal example of this is rational equivalence and Chow groups (see Chapter 1 of \cite{FultInt}).

On more general spaces Allermann and Rau \cite{AllRau} have defined  intersections with tropical Cartier divisors, following a proposal of Mikhalkin in \cite{MikICM}. Moreover, they expressed the diagonal $\Delta \subset \R^n \times \R^n$ as a product of tropical Cartier divisors, and using a procedure analogous to classical geometry may intersect any two cycles in $\R^n$; first by intersecting their Cartesian product with the diagonal in $\R^n \times \R^n$ and then taking the pushforward of the result back to $\R^n$. 
It has been shown independently in \cite{Katz}, \cite{Rau} that when the ambient space is  $\R^n$ this intersection product coincides with the stable intersection mentioned above. 
Once again in Allermann and Rau's theory, the product is defined on the level of tropical cycles, there is no need to pass to equivalence classes. 

The same phenomenon is true of the product on matroidal fans to be defined here. For a matroidal fan, $V \subset \R^n$ the product of two cycles $A, B \subset V$ is a well defined tropical cycle of the expected dimension contained in the fan $V$.   
As expected, the product  is commutative, distributive and associative, 
see Proposition \ref{propMatInt}. The same proposition  proves that the product is compatible with intersections of Cartier divisors from \cite{AllRau}  and \cite{MikICM}. 
%Moreover, the weights of the facets of the product are determined by the local structure of $A, B$ and $V$, see Corollary \ref{defTransInt} and Proposition \ref{proploc}. %Because of this local behavior of the product it is possible to extend it  to more general tropical spaces which are locally built from matroidal fans. 
%As mentioned above, matroidal fans are presently the agreed upon local building blocks of non-singular tropical varieties, so this intersection product may be extended to tropical cycles contained in abstract non-singular tropical varieties. 

The method used here to construct the intersection product on cycles in a matroidal fan is similar in  style to \textit{moving lemmas} from classical algebraic geometry. This  one approach to  classical intersection theory begins with a notion of equivalence of cycles (such as rational equivalence), then  given two cycles $X, Y \subset W$, one shows that there exists a class $X^{\prime}$ rationally equivalent to $X$ which intersects $Y$ properly. Naively speaking, many tropical cycles contained in a matroidal fan may not ``move" on their own. In \cite{MikICM},  there is an example of a rigid tropical cycle contained in a two dimensional matroidal fan in $\R^3$. This line and fan make an reappearance here in Figure \ref{intEx} and Example \ref{negCurve}.  The idea is to construct a procedure which allows us to ``split", instead of move, the tropical cycles into a sum in such a way that the intersection product on the components may be defined. 
The technique used here to construct this splitting comes from tropical modifications.

Tropical modifications, introduced by Mikhalkin in  \cite{MikICM},  are a simple yet powerful  tool in tropical geometry. 
Working over a field $\textbf{K}$,  if  $\textbf{V} \subset \textbf{K}^n$ is an algebraic variety and $\textbf{f}$ a non-singular regular function on $\textbf{V}$ with divisor $\textbf{D} = \text{div}_{\textbf{V}}(\textbf{f})$, the graph of $\textbf{f}$ gives an embedding of $\textbf{V}$ in $\textbf{K}^{n+1}$, with the image of $\textbf{D}$ being contained in the hyperplane $\{z_{n+1} = 0 \}$. This does not correspond to  a very interesting operation classically, however performing the  analogous procedure on tropical varieties produces a polyhedral complex with different topology.
Often we work in  $\R^n$ which is the tropicalisation of the torus $(\textbf{K}^*)^n$. Performing the same procedure as above for a variety $\textbf{V}$ in  $(\textbf{K}^*)^n$, the graph of a non-singular regular function $\textbf{f}$ restricted to $\textbf{V}$ gives an embedding of  $\textbf{V} \backslash \textbf{D}$ to $(\textbf{K}^*)^{n+1}$. 
Given a tropical variety $C \subset \R^n$ and a tropical function $f$ on $\R^n$ the elementary open modification  of $C$ along $f$ should be thought of simply as a reembedding of $C$ with the divisor of $f$ removed. 
 
As mentioned previously, a $k$-dimensional matriodal fan in $\R^n$ may be obtained from $\R^k$ by a sequence of elementary open tropical modifications   along functions with matroidal divisors. 
 Given tropical cycles $A, B$ in a matroidal fan $V \subset \R^n$ we may use this to express the product of two cycles $A, B \subset V$ as a sum of products in different matroidal fans $V^{\prime}$ and $D \times \R$, where $V^{\prime}$ is of lower codimension than $V$ and $D$ is of lower dimension. This procedure is repeated until the intersection is reduced to a  sum of intersections in $\R^k$ where  stable intersection may be applied.   %A similar sort of  double induction procedure on matroids 
 % is used to prove Orlik-Solomon's combinatorial description of the cohomology of the complement of a hyperplane arrangement on $\CP^n$.

% First we relate tropical modifications to well-known operations in matroid theory, known as deletion and contraction. Let  $\tilde{V}$ be a matroidal fan with corresponding matroid $M$. If $\tilde{V}$ is a elementary modification of $V$ along a divisor $D$, then for some element $i$ in the ground set of $M$, the deletion $M \backslash i$ corresponds to $V$ and the restriction $M / i$ corresponds to $D$. 

\vspace{0.5cm}
%:Introduction
The contents of the paper are as follows.
Section  \ref{definitions} reviews   the definitions of tropical cycles, regular and rational functions, tropical modifications/contractions and divisors in $\R^n$ from \cite{MikICM} and \cite{AllRau}. Also we introduce their generalizations to $\T^n = [-\infty, \infty)^n$. In Subsection \ref{sec:Berg}, the connection between tropical modifications and operations  of matroid theory are established. Here, we work in tropical projective space which allows us  to define the Bergman fan of a matroid with loops.  An interesting discovery in this section is the need for \textbf{non-regular} modifications to produce matroidal fans, even in the case of realisable matroids, see Example \ref{M0n}. 

In Section \ref{Intersections}, tropical  modifications and contractions are used to construct an intersection product  on cycles in a matroidal fan. Again the idea is to split the cycles  by using tropical modifications, and then to give the product on the components. Much of the work of this section is devoted to showing that this product is well-defined.  In this section we also show that the product is associative, distributive,  commutative and behaves as expected with divisors, Proposition \ref{propMatInt}.

Section \ref{sec:surfaces} studies the case of one dimensional fan cycles in two dimensional matroidal fans. % simplifies the intersection product for fan curves in two dimensional matroidal fans. 
Firstly, Proposition \ref{surface} simplifies the intersection product in this case. 
Next if two one cycles contained in a two dimensional matroidal fan  are also matroidal, Theorem \ref{thm:linesflats}  describes the intersection product of these cycles in terms of the lattice of flats of the corresponding matroids. 
Finally, using tropical modifications, Theorem \ref{realisable} provides an obstruction to realising effective one dimensional tropical cycles in two dimensional fans  by classical algebraic curves in planes. For instance, this shows that the tropical cycle $B$ from Example \ref{negCurve} is not realisable. 
However, there are tropical curves in surfaces which are known to not be realisable but which are not obstructed by this theorem. 

The author is grateful to Beno\^it Bertrand, Erwan Brugall\'e, Grigory Mikhalkin and Johannes Rau for many helpful discussions and comments. 

\end{section}

\begin{section}{Preliminaries}\label{definitions}

\begin{subsection}{Tropical cycles in $\R^n$}
%:Tropical affine space and cycles
Tropical cycles in $\R^n$ have  been presented in various places, \cite{AllRau}, \cite{KatzTool},  \cite{MikICM}, \cite{FirstSteps}. We review the definitions here for completeness and to ease the generalisations to cycles in $\T^n$.  First we give a  summary of  the necessary terminology. 

A polyhedral complex  $P$ in $\R^n$ is a finite collection of polyhedra containing all the faces of its members and the intersection of any two polyhedra in $P$ is a common face. We say a polyhedral complex $P$ is \textbf{rational} if every face in $P$ is defined by the intersection of half-spaces given by equations $\langle x, v \rangle \leq a $ where  $a \in \R^n$  and $v \in \Z^n \subset \R^n$. The \textbf{support} $|P|$ of a complex is the union of all polyhedra in $P$ as sets, and $P$ is pure dimensional if $|P|$ is. A \textbf{facet} of $P$ is a face of top dimension.  Further, a polyhedral complex $P$ is \textbf{weighted} if each facet $F$ of $P$ is equipped with a weight $w_F \in \Z$. 
A polyhedral complex $P_1$ is a \textbf{refinement} of a complex $P_2$ if their supports are equal and every face of $P_2$ is a face of $P_1$.
For a complex $P$ denote by $P^{(k)}$ the $k$-\textbf{skeleton} of $P$, meaning the union of all faces of $P$ of dimension $i \leq k$.

\begin{definition}\label{bal}
A pure dimensional weighted rational polyhedral complex  $C \subset \R^n$ is balanced if it satisfies the following condition on every codimension one face $E \subset C$: Let  $F_1, \dots F_s$ be the facets adjacent to $E$ and $v_i$ be a primitive integer vector such that for an $x \in E$, $x + \epsilon v_i \in F_i$ for some $\epsilon >0$. Then, 
%generating $F_i$ along with $E$. Then, 
$$ \sum_{i = 1}^s w_{F_i}v_i,$$
is parallel to the face $E$, where $w_{F_i}$ is the weight of the facet $F_i$, see the left hand side of Figure \ref{fig:cycles}.
\end{definition}

\begin{definition}   \cite{MikICM}
A tropical $k$-cycle $C \subset \R^n$ is a pure $k$-dimensional weighted,  rational, polyhedral complex satisfying the balancing condition.
\end{definition}

A tropical $k$-cycle is \textbf{effective} if all of its facets have positive weights.  

\begin{definition}
Given two tropical cycles $A, C \subseteq \R^n$ we say $A$ is a subcycle of $C$ if $|A| \subseteq |C|$  and every open face of $A$ is contained in a single open face of $C$. 
\end{definition}

\begin{remark}
If $A$ is a subcycle of $C$, then there exists a refinement of the polyhedral structure on $C$ so that  $A$ is a polyhedral subcomplex of $C$. Although we will not need to consider this refinement of $C$, the polyhedral structure on $A$ as a subcycle of $C$ will be important. 
\end{remark}

We can define an equivalence relation by declaring a cycle with all facets of weight zero to be equivalent to the empty polyhedral complex. 
The set of tropical $k$ cycles in $\R^n$ modulo this equivalence  will be denoted $Z_k(\R^n)$. 
This set forms a group under the operation of unions of complexes and addition of weight functions denoted by $+$. See \cite{AllRau}, \cite{MikICM} for more details.

As mentioned in the introduction, there have been two approaches to intersections of cycles $\R^n$. Firstly, tropical stable intersection was defined for curves in $\R^2$ in \cite{FirstSteps} and for general cycles by Mikhalkin in \cite{MikICM}. The intersection product  in $\R^n$ of Allermann and Rau  is based on intersecting with the diagonal $\Delta \subset \R^n \times \R^n$, see \cite{AllRau}. The two definitions have been shown to be equivalent in both \cite{Rau}, \cite{Katz}. We review the definition of stable intersection in $\R^n$. 

\begin{definition}\label{stabRn} \cite{FirstSteps},  \cite{MikICM}
Let $A \in Z_{m_1}(\R^n)$ and $B \in Z_{m_2}(\R^n)$, then their stable intersection, denoted $A.B$  is supported on the complex $(A \cap B)^{(k)}$ where $m = m_1 + m_2 - n$ with weights assigned on facets in the following way:

\begin{enumerate}

\item \label{S} If a facet $F \subset (A \cap B)^{k}$ is the  intersection of top dimensional facets $D \subset A$ and $E \subset B$ and $D$ and $E$ intersect transversely, then 
\begin{equation*}\label{transvWeight}
w_{A.B}(F) = w_A (D) w_B (E) [ \Z^n \colon \Lambda_D + \Lambda_E ], 
\end{equation*}
where $\Lambda_D$ and $\Lambda_E$ are the integer lattices parallel to the faces $D$ and $E$ respectively.

\item Otherwise for a generic vector $v$ with non-rational projections and an $\epsilon > 0$, in a neighborhood of $F$, $A_{\epsilon} = A + \epsilon \cdot v$ and $B$ will meet in a collection of facets $F_1 \dots F_s$ parallel to $F$ such that the intersection at each $F_i$ is as in the case $(1)$ above. 
Then we set, 
\begin{equation*}\label{nonTrans}
w_F(A.B) = \sum_{i=1}^s w_{F_i}(A_{\epsilon}. B).
\end{equation*}

\end{enumerate}

\end{definition}

That the formula above is well-defined regardless of choice of the vector $v$ follows from the balancing condition. In fact,  the above weight calculation comes from the fan displacement rule for intersection of Minkowski weights from  \cite{FulStur}, for more details see  \cite{AllRau} or \cite{Katz}.
By the equivalence of stable intersection and Allermann and Rau's intersection product on $\R^n$ shown in  \cite{Rau} \cite{Katz}, the following two propositions can be found in Section 9 of \cite{AllRau}.

\begin{corollary}\cite{AllRau}\label{asscom}
Given $A$,  $B$ tropical cycles in $\R^n$ the following hold, 
\begin{enumerate}
\item $A.B \subset \R^n$ is a balanced tropical cycle.
\item $(A.B).C= A.(B.C) $
\item $A.B = B.A$
\item $A.(B +C) = A.B +A.C$ 
\end{enumerate}
\end{corollary} 

\end{subsection}

\begin{subsection}{Tropical cycles in $\T^n$.}%affine space. }

The tropical numbers $\T = \R \cup \{ -\infty \}$ form a semi-field equipped with  the following operations:
$$``x \cdot y"= x + y \  \mbox{and} \  ``x+y" = \max \{x, y\}.$$ As the multiplicative and additive identity we have  $1_{\T} = 0$,  $0_{\T}= -\infty$ and tropical division corresponds to  subtraction.
We equip $\T^n =   [ -\infty, \infty )^n$ with the Euclidean topology, and will think of it as tropical  affine $n$-space.  It  has a boundary which admits a natural stratification in the following way: Let $H_i = \{ x \in \T^n \ | \ x_i = -\infty\}$, be the $i^{th}$ coordinate hyperplane. Given a subset $I \subseteq [n] =  \{1, \dots , n\}$ denote $H_I = \cap_{i \in I} H_i$, and  $$H_I^{\times}   = \{x \in H_I | x \notin H_J \ I \subset J \}.$$ Then, $$\T^n = \coprod_{\emptyset \subseteq I \subseteq [n]} H^{\times}_I.$$ For every $I \in [n]$, we have $H_I =  \T^{n-|I|}$ and $H^{\times}_I  = \R^{n-|I|}$.  Say a point  $x \in H^{\times}_I \subset \T^n$ is of \textbf{sedentarity} $I$, this is denoted $S(x) = I$. The \textbf{order of sedentarity} of $x \in \T^n$ is the size of $S(x)$ and denoted $s(x)$.

%\begin{definition} 
%The sedentarity of a point $x \in H^{\times}_I \subset \T^n$ is denoted by $S(x)  = I$.   The order of sedentarity of $x$ is $s(x) = |I|$ 
%\end{definition}

We can generalise the definition of cycles in $\R^n$  to $\T^n$ by allowing  cycles contained in the boundary strata of $\T^n$.

\begin{definition}\label{sedSubset}
A subset $B \subseteq \T^n$ is said to be of sedentarity $I$ if it is the topological closure in $\T^n$ of some $B^o \subset H^{\times}_I$.  A tropical $k$-cycle  $C \subseteq \T^n$  of sedentarity $I$ is the closure of a tropical $k$-cycle $C^o \subseteq H^{\times}_I =  \R^{n-|I|} $, see the right hand side of Figure \ref{fig:cycles}. 
\end{definition}

\begin{figure}
a)\includegraphics[scale=0.8]{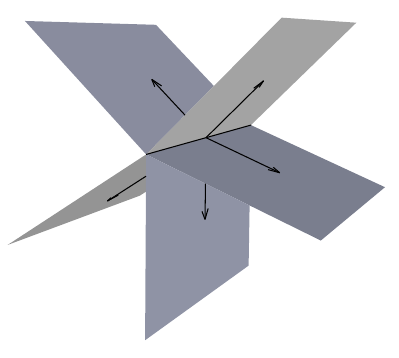}
\put(-40,60){\large{$v_1$}}
\put(-50,113){\large{$v_2$}}
\put(-110,110){\large{$v_3$}}
\put(-120,45){\large{$v_4$}}
\put(-80,40){\large{$v_5$}}
\hspace{0.5cm}
b) \includegraphics[scale=0.9]{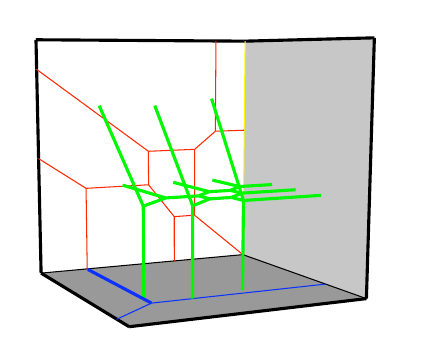}
\put(-23, 15){$x_1$}
\put(-180, 30){$x_2$}
\put(-83, 135){$x_3$}
\vspace{0.3cm}
\caption{a) Balancing condition for a surface b) Cycles of sedentarity in $\T^3$. \label{fig:cycles}}

\end{figure}
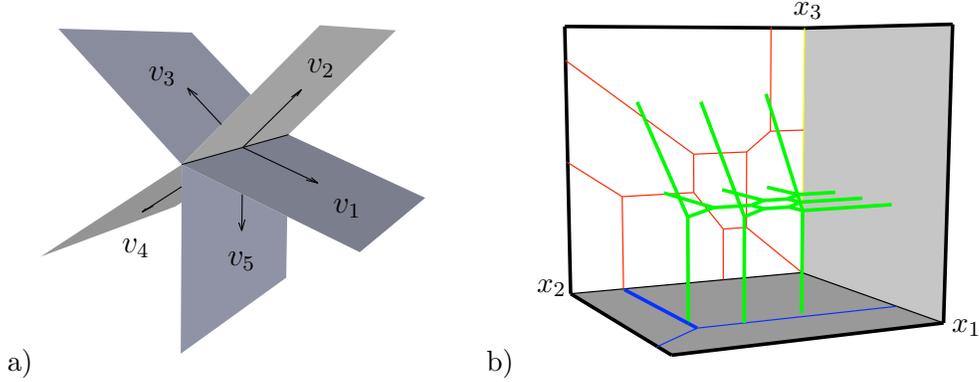

Again let, $ Z_{k, I}(\T^n)$ denote the quotient of the set of all $k$-cycles of sedentarity $I$  by those with all zero weights.  Given two cycles $A, B \in Z_{k, I}(\T^n)$ denote by $A+B$ the closure of $A^o + B^o$ as defined in $\R^{n-{I}}$. Then, $ Z_{k, I}(\T^n) \cong Z_k(\R^{n- |I|})$ and we define, 
$$Z_k(\T^n) =  \bigoplus_{ \emptyset \subseteq I \subseteq [n] } Z_{k, I} (\T^n).$$

Once again, a tropical cycle in $\T^n$ is \textbf{effective} if all of its facets are equipped with positive weights. Also, as in $\R^n$, a   tropical cycle $A \subset \T^n$ is a \textbf{subcycle} of a cycle $C \subset \T^n$ if the supports satisfy $|A| \subseteq |C|$ and every face of $A$ is contained in a face of $C$.

\vspace{0.3cm}

For cycles in $\T^n$ we define their intersection with a boundary hyperplane. Let $e_1, \dots e_n$ denote the standard basis of $\R^n$. 

\begin{definition}\label{cycleIntStratum}
Let $A \subseteq \T^n$ be a $k$-cycle of sedentarity $I$ then 

\begin{itemize}
\item{if  $i \in I$, set $A.H_i = \emptyset$.}
\item{If $I = \emptyset$ then $A.H_i$ is supported on $(A \cap H_i)^{(k-1)}$  with the weight function defined as follows:  Given a facet $F$ of $(A \cap H_i)^{(k-1)}$ it is adjacent to some facets $\tilde{F}_1, \dots , \tilde{F}_s$ of $A$. 
Then, $$w_{A.H_i}(F) = \sum_{l=1}^s w_A (\tilde{F}_l)[ \Z^n : \Lambda_{\tilde{F}_l} + \Lambda_i^{\perp}],$$
where $\Lambda_i^{\perp}= \{ x \in \Z^n \ | \ \langle x, e_i \rangle = 0\}$. }
\item{If $I \not = \emptyset$ and $i \not \in I$ then  $A.H_i$ is the intersection of $A.H_{i \cup I}$ calculated in $H_I = \T^{n-|I|}$ as in the case above.}
\end{itemize}

\end{definition}

Every cycle $A \subseteq \T^n$ can be uniquely decomposed as a sum of its parts of different sedentarity and we extend the above definition to cycles of mixed sedentarity by linearity. 

\begin{proposition}\label{propIntStratum}
Given cycles $A, B \subseteq \T^n$ we have:
\begin{enumerate}
\item $A.H_i$ is a balanced cycle. 
\item $(A+B).H_i = A.H_i + B.H_i$.

\end{enumerate}

\end{proposition}

\begin{proof}
\comment{We must only check balancing at the codimension one faces of $A.H_i$ laying in the interior of a boundary stratum. We will assume that $A$ is of sedentarity $I = \emptyset$. Given a  codimension one face $E \subset A.H_i$, it is adjacent to a collection $\tilde{E_1}, \dots , \tilde{E}_s$ of codimension one faces of $A$. For $M >> 0 $ let $L_M = \{  x \in \T^n \ | \ x_i = -M \}$, then $A.L_M$ is a balanced cycle and has some codimension one faces $E_i$ corresponding to $\tilde{E}_i \cap L_M$ for $1 \leq i \leq s$ where $\tilde{E}_i$ is a codimension one face of $A$. At each of these faces $E_i$ the cycle $A.L_M$ is balanced. \textcolor{red}{All facets $\tilde{F}$ which are adjacent to two or more $\tilde{E}_i$ in $A$ are collapsed to $E$ in $A.H_i$. }
Any other facets adjacent to a $\tilde{E}_i$ get sent to facets of $A.H_i$ adjacent to $E$. }

For the balancing condition assume that $A$ is of sedentarity $\emptyset$ and let $E \subseteq  A.H_i$ be a face of codimension one which is in the interior of a face of $\T^n$ of sedentarity $\{i\}$. Let $\tilde{E}_j$ denote the faces of codimension one of $A$ and of sedentarity $\emptyset$ which are adjacent to $E$. For $M >>0$ let $L_M = \{  x \in \T^n \ | \ x_i = -M \}$ then $\tilde{E}_j \cap L_M$ is in $A.L_M$  and $A.L_M$ is balanced at $\tilde{E}_j \cap L_M$. This means that 

\begin{equation}\label{eq:boundaryInt}
\sum_{\tilde{E_j} \subset \tilde{F}} w_{A.L_M}(\tilde{F} \cap L_M)v_{\tilde{F}} = \sum_{\tilde{E_j} \subset \tilde{F}} w_A (\tilde{F})[ \Z^n : \Lambda_{\tilde{F}_l} + \Lambda_i^{\perp}]v_{\tilde{F}}= 0, 
\end{equation}

%$$\sum_{\tilde{E_j} \subset \tilde{F}} w_{A.L_M}(\tilde{F} \cap L_M)v_{\tilde{F}} = \sum_{\tilde{E_j} \subset \tilde{F}} w_A (\tilde{F})[ \Z^n : \Lambda_{\tilde{F}_l} + \Lambda_i^{\perp}]v_{\tilde{F}}= 0, \ \ \ \ \  (\ast)$$  where $v_{\tilde{F}} $ is the primitive integer vector orthogonal to $\tilde{E}$ generating $\tilde{F}$. 

Let $\pi_I: \R^n \longrightarrow \R^{n-|I|}$ be the linear projection with kernel $<e_i \ | \ i \in I >$.
Then a facet $\tilde{F} \supset \tilde{E}_j$ is adjacent to a face $F \supset E$ if and only if  $\pi_{I \ast}(v_{\tilde{F}}) = v_{F}$ where $v_F$ is the primitive integer vector in $\R^{n-|I|}$ orthogonal to $E$ generating $F$. Applying $\pi_{I \ast}$ to (\ref{eq:boundaryInt}) and taking the sum over all $\tilde{E}_j$ adjacent to $E$ we obtain balancing at $E$.

\comment{The image of all of the $\tilde{E}_i \cap L_M$'s and $\tilde{F} \cap L_M$ for all $\tilde{F}$'s adjacent to the $\tilde{E}_i$ under the projection $\pi_I$ corresponds to the facets adjacent to $E$ in the boundary stratum. 
Each term in the sum from \ref{cycleIntStratum} is the same as the weight of $\tilde{F}_i \cap L_M$ in $A.L_M$. }

When $A$ and $B$ are of equal sedentarity distributivity follows from the additivity of the weight function.  For cycles of mixed sedentarity the intersection is defined by extending the product linearly, so the statement is trivial. This completes the proof. 
\end{proof}

\end{subsection}

\begin{subsection}{Tropical functions, modifications and divisors.}\label{sec:mods}
%:functions, modifications and divisors

\begin{definition}
Let $U$ be an open subset of $\T^n$ and let  $S(U) = \bigcup_{x \in U} S(x) \subset [n]$.  A tropical regular  function $f: U \longrightarrow \T$ is a tropical Laurent polynomial $f(x) = `` \sum_{\alpha \in \Delta} a_{\alpha}x^{\alpha} "$   where $\Delta \subset \Z^n$ is such that for all $\alpha \in \Delta$,  $\alpha_i \geq 0$ if $i \in S(U)$. 
\end{definition}

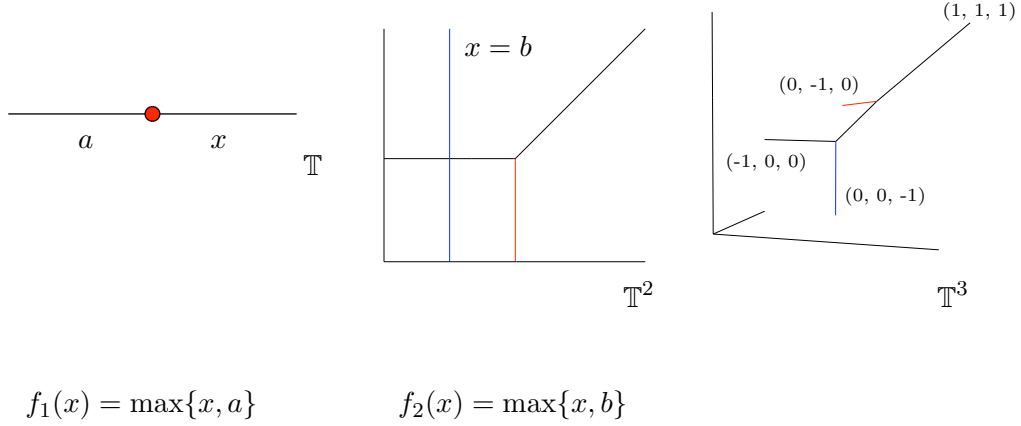
\begin{figure}
\vspace{-1cm}
\includegraphics{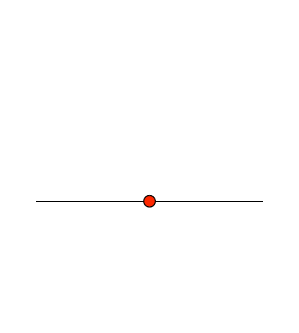}
\put(-100,50){$a$}
\put(-50, 50){$x$}
\put(-15, 40){$\T$}
\put(-120, -50){$f_1(x) = \max \{x, a\}$}
\hspace{0.2cm}
\includegraphics[scale=0.5]{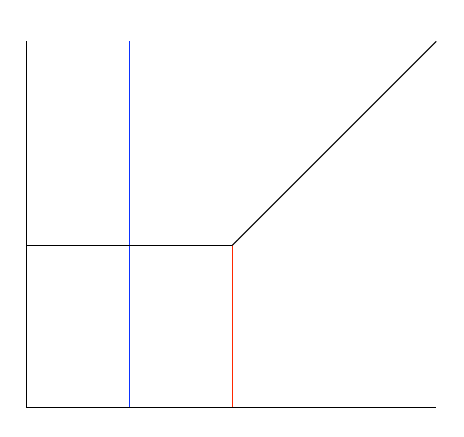}
\put(-100, -50){$f_2(x) = \max \{x,b\}$}
\put(-75, 85){$x=b$}
\put(-15, -10){$\T^2$}
\hspace{0.2cm}
\includegraphics[scale=0.6]{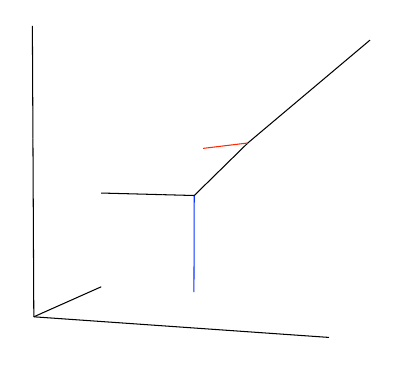}
\put(-18, 100){\tiny{(1, 1, 1)}}
\put(-55, 30){\tiny{(0, 0, -1)}}
\put(-80, 72){\tiny{(0, -1, 0)}}
\put(-100, 43){\tiny{(-1, 0, 0)}}
\put(-20, -10){$\T^3$}
\vspace{0.5cm}
\caption{Two modifications of a tropical line}
\end{figure}

A tropical regular function is a  piecewise integral affine, convex function, whose graph is  a finite polyhedral complex.
Suppose $U \subseteq \R^n \subset \T^n$ then every regular function on $U$ can be expressed as a tropical Laurent polynomial $f(x) = `` \sum_{\alpha \in \Delta} a_{\alpha} x^{\alpha} "$.  % and every such monomial is an invertible function. 
If $U$ contains a point $x$ for which $x_i = -\infty$ for some $i$, then $``1/x_i" = -x_i = \infty \notin \T$.  Distinct tropical polynomials may represent the same functions as some monomials may be redundant.  Let  $\Oc_{\T^n}(U)$ denote  the semi-ring of regular functions on $U$ and $\Oc_{\T^n}$ the regular functions on $\T^n$. 

Tropical division corresponds to subtraction and so a rational function is of the form $ h = `` f / g" = f  - g$ where $g \neq -\infty$. On $\R^n \subset \T^n$ such a function is always defined since it is the difference of two continuous functions.  At the boundary of $\T^n$ where the function may take values $\pm \infty$ there may be a codimension two locus where the function is not defined. For example the function $f (x) = ``\frac{x_1}{x_2}"$ on $\T^2$ at the point $(-\infty, -\infty)$. 
 We denote the rational functions on $\T^n$ by $\K_{\T^n}$.  Given a tropical cycle $C \subseteq \T^n$ (or $C \subseteq \R^n$) regular functions and rational functions on $C$, denoted $\Oc_C$ and $\K_C$ respectively, are obtained by restriction of $\Oc_{\T^n}$ and $\K_{\T^n}$ (or $\Oc_{\R^n}$ and $\K_{\R^n}$) .

\vspace{0.5cm}

Given a cycle $C \subseteq \T^n$  we may consider the graph $\Gamma_f(C) \subset \T^{n+1}$ of a function $f \in \Oc_C$ restricted to $C$.  The graph $\Gamma_f(C)$ is still a rational polyhedral complex, and it inherits weights from  $C$. Since $f$ is piecewise affine $\Gamma_f(C)$ is not necessarily balanced. At any unbalanced codimension one face $E$ of $\Gamma_f(C)$ we may attach the closed facet $F_E$, generated by $E$ and  the direction $-e_{n+1}$, more precisely, 
 $$F_E= \{ (x, c) \ | \ x \in \overline{E}, c \in (x, -\infty] \} .$$ 
 Moreover, there exists a unique integer weight on $F_E$ such  that the resulting complex is now balanced at $E$.
Let the \textbf{undergraph} of $\Gamma_f(C)$ be the weighted rational polyhedral complex 
$$\mathcal{U}(\Gamma_f(C)) = \bigcup_{\substack{E \subset \Gamma_f(C) \\ codim(E) = 1}} F_E,$$ 
with weights described above. 
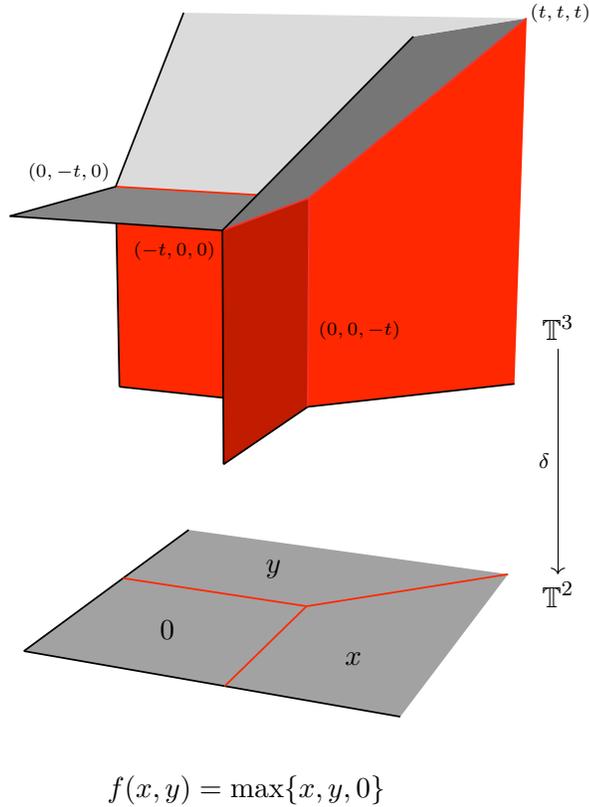
\begin{figure}
\includegraphics[scale=1.3]{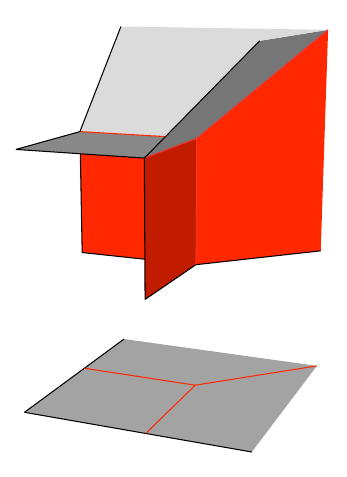}
\put(-150,45){0}
\put(-110, 70){$y$}
\put(-80, 35){$x$}
\put(-170, -15){$f(x, y) = \max \{x, y, 0 \}$}
\put(-10, 160){$\xymatrix{\T^3 \ar[ddd]_{\delta} \\ \\ \\ \T^2} $}
\put(-200, 220){\tiny{$(0,-t,0)$}}
\put(-10, 280){\tiny{$(t,t,t)$}}
\put(-160, 190){\tiny{$(-t,0,0)$}}
\put(-90, 160){\tiny{$(0,0, -t)$}}
\vspace{0.5cm}
\caption{A modification $P$ of  the tropical affine plane $\T^2$}
\label{Plane}
\end{figure}
Finally, the weighted complex 
 $$\tilde{C} = \Gamma_f(C) \cup \mathcal{U}(\Gamma_f(C)), $$
is a tropical cycle. Let $\delta:\T^{n+1} \longrightarrow \T^n$ be the linear projection with kernel generated by $e_{n+1}$. 

%Then the elementary  modification of $C$ along $f$ is, 
\begin{definition}\label{PmodDef}
Given a%n effective 
cycle $C \subseteq \T^n$ and a regular function $f \in \Oc_C$,  %the cycle $\tilde{C} \subset \T^{n+1}$
 the regular elementary modification of $C$ along the function $f$ is 
$$\delta:\tilde{C} \longrightarrow C.$$
\end{definition}

 Often the term ``elementary modification" will be used to denote only the cycle $\tilde{C}$, in this case the existence of the map $\delta:\tilde{C} \longrightarrow C$ and function $f$ is implied.  The  cycle $C$ will also sometimes be referred to as the contraction of $\tilde{C}$. This notation is similar in style to that used for blow-ups in
 classical algebraic geometry. 
 
 A \textbf{regular modification}, respectively \textbf{regular contraction}, is any composition of regular elementary modifications, respectively contractions. 
Using modifications we define the divisor of a function on a cycle $C \subseteq \T^n$. 
\begin{definition}\label{divTn}
Let $f, g: \T^n \longrightarrow \T$ be regular functions and suppose $g \neq 0_{\T}$ and let $C \subset \T^n$ be a cycle and $\delta_f:\tilde{C}\longrightarrow C$ the elementary modification of $C$ along the function $f$. Then, 
\begin{enumerate}
\item{$ \mbox{div}_{C}(f)  = \delta_f(\tilde{C}.H_{n+1})$}
\item{If  $h =``f/g"$ then 
$\mbox{div}_{C}(h)  =  \mbox{div}_{C}(f)  -  \mbox{div}_{C}(g) .$ }
\end{enumerate}
\end{definition}
Given a regular elementary modification $\delta: \tilde{C} \longrightarrow C$ along a function $f$, we say that $\text{div}_C(f)$ is the \textbf{divisor of the modification}. 

All of the above definitions given for cycles in $\T^n$ restrict to cycles in $\R^n$, and we may modify cycles in $\R^n$ along regular functions in $\Oc_{\R^n}$, not just functions on $\T^n$. In particular we have 
$\text{div}_{C^o}(f) = \text{div}_C(f) \cap \R^n$. This definition coincides with the definition of divisors from \cite{AllRau} and \cite{MikICM}.  From the definition of divisors we notice that a tropical invertible function on  $\T^n$ is tropical multiplication by a scalar $x \in \T^{\times} = \R$ (so addition), and a tropical invertible function on $\R^n$ is a Laurent monomial. 
\begin{proposition} \label{StabDiv} \cite{AllRau}, \cite{Katz}
Given tropical rational functions $f, g \in \K_{\R^n}$ and tropical cycles $A, B \subset \R^n$ 
\begin{enumerate}
\item $\mbox{div}(f)_{\R^n}.A = \mbox{div}_A(f)$
\item $\mbox{div}_{A+B}(f) = \mbox{div}_A(f) + \mbox{div}_B(f)$
\item $\mbox{div}_A(``f \cdot g") = \mbox{div}_A(f) + \mbox{div}_A(g)$ 
\end{enumerate}
\end{proposition}
 %
 %
 %s
 The following proposition will be needed later on in Proposition \ref{pushModCon}.
\begin{proposition}\label{propDiv}
For functions $f, g \in \K_{\T^n}$ and cycles $A, B \subset \T^n$.  We have,
$$\mbox{div}_{A+B}(f) = \mbox{div}_A(f) + \mbox{div}_B(f).$$
%\textcolor{red}{\item $\mbox{div}_A(``f \cdot g") = \mbox{div}_A(f) + \mbox{div}_A(g)$ }
\end{proposition}
\begin{proof}
This follows directly from Proposition \ref{propIntStratum} and Definition \ref{divTn}. 
\end{proof}
Following part $(1)$ of Proposition \ref{StabDiv} we have:
\begin{corollary}\label{regEff}
If $f \in \Oc(\R^n)$ then %$\text{div}_{\R^n}(f).C$ 
$\text{div}_{C}(f)$ 
is effective for every effective  cycle $C \subset \R^n$.
\end{corollary}
Regular modifications can be generalized to rational functions with effective divisors on effective  cycles $C$. The construction of the modification $\tilde{C}$ is the same as in the regular case. The resulting cycle $\tilde{C}$ is effective since the weights of the facets of $\mathcal{U}(\Gamma_f(C))$ correspond to the weights of the facets of $\text{div}_C(f)$,  which is assumed to be effective.  

\begin{definition}\label{modDef}
Given an effective cycle $C \subset \T^n$ and a rational  function $f \in \K_C$ such that $\mbox{div}_C(f)$ is effective,  the elementary modification of $C$ along the function $f$ is $$\delta:\tilde{C} \longrightarrow C,$$
where again $\tilde{C} = \Gamma(C) \cup \mathcal{U}(\Gamma_f(C))$. 
%  elementary modification of $C$ along the function $f$. The linear projection  $\delta: \tilde{C} \longrightarrow C$ is a  elementary contraction  map and $C$ is called a elementary contraction of $\tilde{C}$. 
\end{definition}

A \textbf{modification}, respectively \textbf{contraction}, is any composition of elementary modifications, respectively contractions.  
For an elementary modification of a cycle in $\T^n$ we can define pullback and pushforward maps on subcycles.
\begin{definition}\label{pushModCon}
Let $C \subset \T^n$ be an effective cycle, $f \in \K_C$ be a function with effective divisor on $C$ and $\delta: \tilde{C} \rightarrow C$ be the elementary modification along $f$. We define the following:
\begin{enumerate}\label{pushpull}
\item{The pushforward map of cycles, $\delta_{\ast}: Z_k(\tilde{C}) \rightarrow Z_k(C)$  is given by  $\delta_{\ast}A = \delta(A)$ with weight function, 
$$w_{\delta_{\ast}A}(F) = \sum_{F_i \subset A,  \delta(F_i) = F} w_A(F_i)[\bar{\Lambda}_{F_i} : \Lambda_{F}],$$ where $\bar{\Lambda}_{F_i}$ is the image under $\delta$ of the integer lattice generated by $F_i$ and  $\Lambda_{F}$ is the integer lattice generated by $F$. }
\item{The pullback map of cycles  $\delta^{\ast}: Z_k(C) \rightarrow Z_k(\tilde{C})$. For a cycle  $A \in Z_k(C)$, $\delta^{\ast}A$ is the modification of $A$ along the function $f$.}
% $\delta^{\ast}A$ is the cycle obtained from the graph  $\{(x, f(x))  \ | \ x \in A \}$ by adjoining the unique collection of weighted  facets in the direction $-e_{n+1}$ so that the resulting complex is balanced. }
\end{enumerate}
\end{definition}

Clearly the cycle $\delta^{\ast}A$ is contained in $\tilde{C}$. Notice that $\delta_{\ast}\delta^{\ast}A = A$ but $\delta^{\ast}\delta_{\ast}A$ is not always equal to $A$.   Also, the pullback of an effective cycle may not be effective if the modification of $C$ is given by  a rational function, an example of this can be found in \cite{AllRau} and \cite{MikICM}. Moreover, since the definition of the weight function on the pushforward is additive   $\delta_{\ast}$ is a homomorphism.  The pullback map is also a group homomorphism by 
Proposition \ref{propDiv}.

\end{subsection}

\begin{subsection}{Bergman fans of matroids and tropical modifications}\label{sec:Berg}
%:Bergman fans

Here we study tropical modifications in relation to Bergman fans of  matroids. This section provides a correspondence between tropical modifications and existing constructions in matroid theory.  
There are many equivalent ways of describing a  matroid, here we will most often use the rank function. So we write a matroid as $M = (E, r)$ where $E = \{0, \dots N\}$ is the ground set  and $r$ is a rank function,  $r: \mathcal{P}(E) \longrightarrow \N \cup \{0\}$,  satisfying certain axioms, see \cite{Ox}. 
The \textbf{flats} of a matroid are the subsets $F \subset E$ such that the rank function satisfies $r(F)< r(F \cup i)$ for all $i \in E$ not contained in $F$. By convention,  the ground set $E$  is also a flat. The flats of a matroid $M$ form a lattice, which we will denote $\Lambda_M$. 

In this section the focus is on the following basic concepts from matroid theory and their connections to tropical modifications.  We include their definitions for the reader not familiar with matroid theory. Again, for a comprehensive introduction to the subject see \cite{Ox}.

 \begin{definition}
Let $M= (E, r)$ be a matroid where  $E = \{0, \dots , n\}$ %is the ground set and $r: \mathcal{P}(E) \longrightarrow \N \cup \{0\}$ is the rank function 
and $e \in E$, then 

 \begin{enumerate}
\item{The deletion with respect to $i$, $M \backslash i$ is the matroid $(E \backslash i, r|_{E \backslash i})$.}
\item{The restriction with respect to $i$, $M / e$ is the matroid $(E \backslash i, r^{\prime}) $ where $r^{\prime}(I) = r(I \cup i) - r(i)$. }
\item{A matroid $Q$ is a single element  quotient of $M$ if  there exists a matroid $N$ on a ground set $E^{\prime} = E \cup i^{\prime}$ such that $N \backslash i^{\prime} = M$ and $N / i^{\prime} = Q$, and $N$ is called a single element extension of $M$.}
\end{enumerate} 
 \end{definition}

Deletions and restrictions can be performed  with respect to a subset $I \subset E$, these will be denoted $M \backslash I$ and $M / I$ respectively. Also  a matroid $Q$ will be called a quotient of $M$ if there is a matroid $N$ with ground set $E \cup F$ such that $N \backslash F = M$ and $N / F = Q$, and $N$ will simply be called an extension. 

We wish to consider a projective version of the Bergman  fan of a matroid $M$ contained in tropical projective space. 
%also we will consider the topological closure of these complexes in tropical projective space. 
%
\begin{definition}\cite{MikICM}
Tropical projective space is 
$$\TP^n = (\T^{n+1} \backslash (-\infty, \dots , -\infty) )/ (x_0, \dots , x_n) \sim (x_0 +\lambda, \dots , x_n + \lambda)$$
for $\lambda \in \R$. 
\end{definition}
Tropical projective space is topologically the $n$-simplex. We can equip $\TP^n$ with tropical homogeneous coordinates $[x_0:  \dots : x_n]$ similarly to the classical setting. It may be covered by $n+1$ charts 
$U_i = \{ [x_0: \dots :x_n] \ | \ x_i = 0\} = \T^n$. Moreover, the boundary of $\TP^n$ is stratified;
%in a way similar to $\T^n$. 
given $\emptyset \neq I \subseteq \{0, \dots , n\}$ we have a face of the $n$-simplex corresponding to the subset of $\TP^n$ where $x_i = -\infty$ in homogeneous coordinates. Moreover such a face is isomorphic to $\TP^{n- |I|}$. Similar to $\T^n$, a tropical cycle $A$ in $ \TP^n$ is the closure of a cycle in $\R^{n-|I|} \subset \TP^{n-|I|}$ identified as one of the boundary strata of $\TP^n$. 

We now  review of the construction of the Bergman fan of $M$, denoted $B(M)$, in terms of the lattice of flats from \cite{ArdKli}. Recall that a \textbf{loop} of a matroid is an element $i \in E$ that is not contained in any basis. First assume that $M$ is loopless, meaning it contains no loops. 
For $1 \leq i \leq n$ set $v_i = -e_i$ and $v_{0} = \sum_{i=1}^n e_i$, where $e_1, \dots, e_n$ are the standard basis vectors of $\R^n \subset \T^n$. Let $\Lambda_M$ denote the collection of flats of $M$. For every chain $\emptyset \neq F_1 \subset \dots \subset F_k \neq E$  in  $\Lambda_M$ we have a corresponding $k$ dimensional cone in $B(M)$  given by the positive span of $ v_{F_1} , \dots ,  v_{F_k} $ where $ v_{F_l} = \sum_{i \in F_l} v_i$. Finally, $B(M)$ is the closure in $\TP^n$ of the union of all such polyhedral cones. This is the \textbf{fine} polyhedral structure on $B(M)$ as defined in \cite{ArdKli}. 
%As defined above $B(M)$ is equipped with the refined polyhedral structure that we will often forget and only consider the coarse structure, see \cite{ArdKli}. 
This construction is a projectivisation of the definitions given  in \cite{SturmFeich}, \cite{ArdKli} up to a reflection caused by the use of the \textit{max} convention instead of \textit{min}. 

\begin{exa}
A geometric example of matroids is to consider a  hyperplane arrangement $\mathcal{A} = \cup_{i=0}^{n} L_i$ in  $\CP^k$. We can define a matroid on $E= \{0, \dots, n\}$, by the rank function $r(E^{\prime}) = \mbox{codim} (\cap_{i \in E^{\prime}} L_i)$. Suppose the rank of the associated matroid is $k+1$ which is equivalent to $\cap_{i=0}^nL_i = \emptyset$.   Each hyperplane is given by a linear form $f_i$. 
 Using these forms we can define a map:
\begin{align*}
\phi: \CP^k & \longrightarrow \CP^n \\
 x \ & \mapsto [f_0(x): \dots : f_n(x)]
\end{align*}

If $M$ is loopless then $\phi(\CP^k) \cap (\C^{\ast})^n = \CP^k \backslash \mathcal{A}$, the complement of the hyperplane arrangement. The logarithmic limit set of $\phi(\CP^k) \cap (\C^{\ast})^n$  is the Bergman fan of the associated matroid $M_{\mathcal{A}}$, see \cite{Berg}, \cite{SturmPoly} for more details.

 Again if $M_{\mathcal{A}}$ is the matroid arising from a hyperplane arrangement $\mathcal{A}$ we can interpret the above operations geometrically. The deletion, $M_{\mathcal{A}} \backslash i$, corresponds to the arrangement given by removing the $i^{th}$ hyperplane, 
$$\mathcal{A}^{\prime} = \mathcal{A} \backslash L_i,$$ and the restriction $M_{\mathcal{A}}  / i$ is the  arrangement on $\CP^{n-1}$ obtained by restricting the arrangement $\mathcal{A}$ to $L_i$. % induced arrangement on $L_e$, 
$$\mathcal{A}^{\prime \prime} =  \cup_{j \neq i} (L_i \cap L_j).$$ For more on this see Section $1$ of  \cite{OrlikTerao}.

If $i$ is a loop then $\mbox{codim}(i) = 0$ meaning $L_i$ is the degenerate hyperplane defined by the linear form $f_i = 0$, and so  $\phi(\CP^k)$ is contained in the $i^{th}$ coordinate hyperplane of $\CP^n$. 
From this we next define the Bergman fan in $\TP^n$  for a matroid with loops.
\end{exa}

\begin{definition}
Given a matroid $M = (E, r)$,   let $I \subset E$ denote its collection of loops.  Then the complex $B(M)$ is contained in the boundary of $\TP^n$ corresponding to $x_l= -\infty$ for all $l \in I$ and is equal to $B(M \backslash I) \subseteq \TP^{n - |I|}$. 
\end{definition}

By the following lemma all quotients of a matroid $M$ can be represented geometrically as Bergman fans of matroids which are polyhedral subcomplexes of $B(M)$. Here we use the fine polyhedral subdivision from \cite{ArdKli} as described in this section.
\begin{lemma}\label{quot}
A matroid $Q$ is a quotient of $M$ if and only if $B(Q) \subseteq \TP^n$ is a polyhedral subcomplex of $B(M) \subseteq \TP^n$.  
\end{lemma}
 \begin{proof}
We may assume $M$ is loopless and that $Q$ is a single element quotient of $M$, since every quotient can be formed by a sequence of single element quotients.  Moreover, by Proposition 7.3.6 of \cite{Ox} $Q$ is a quotient of $M$ if and only if $\Lambda_Q \subseteq \Lambda_M$. 
So supposing $Q$ is loopless, the lemma follows immediately from the above statement and the construction  of $B(M)$ in terms of the lattice of flats. If $Q$ contains loops $L \subset E$, %  we must show that $B(Q)$ is contained in the boundary of $B(M)$ contained in the boundary strata of $\T^n$ corresponding  to $x_l = -\infty$ for $l \in I$ if and only if $\Lambda_Q \subset \Lambda_M$. 
then $B(Q)$ is contained in the boundary stratum of $\TP^n$ corresponding to $x_l = -\infty$ for all $l \in I$. A face of $B(Q)$  corresponding to a chain of flats  $I = F_0 \subset F_1 \subset \dots \subset F_{s} \neq \{0, \dots , n\}$ of $\Lambda_Q$, is contained in the boundary of $B(M)$ if and only if the same chain is a chain in $\Lambda_M$ and the lemma is proved. \end{proof}

 The next  proposition relates tropical modifications, contractions and divisors to matroid extensions, deletions and restrictions, respectively.
Recall that an element $i \in E$ is a \textbf{coloop} of a matroid $M = (E, r)$ if $i$ is contained in every basis of $M$, i.e. $i \in B$ for every $B \subset E$ for which  $r(B) = |B| = r(E)$. If a matroid $M$ contains $m$ coloops then the corresponding Bergman fan $B(M) \subset \TP^n$ contains an $m$ dimensional subspace of  $\R^n$.

\begin{proposition}\label{matroidMod}
Let $M$ be a rank $k+1$ matroid on the ground set $E = \{ 0, \dots , n \}$. Suppose $i \in E$ is neither a loop nor a co-loop, then in every chart $U_j = \{ x \in \TP^n \ | \ x_j \neq -\infty \}  = \T^n \subset \TP^n$ there is an elementary  tropical modification $$\delta_j: B(M) \cap U_j \longrightarrow B(M \backslash i) \cap U_j$$ with corresponding divisor $B(M / i) \cap U_j$. 
\end{proposition}

\begin{proof}
For the lattice of flats of deletions and restrictions we have:
\begin{align*}
\Lambda_{M \backslash i} = \{F \subseteq E \backslash i \ | \ F \ \mbox{or} \ F \cup i \ \mbox{is a flat of M}\} \\
 \Lambda_{M / i} =  \{F \subseteq E \backslash i \ | \ F \cup i \ \mbox{is a flat of M}\}.
\end{align*}
Let $\delta_i: \T^n \longrightarrow \T^{n-1}$ be the projection in the direction of $e_i$. Then the image under 
$\delta_i$ of a $k$-dimensional cone of $B(M) \cap U_j$ corresponding to a chain of flats   $F_1 \subset \dots \subset F_k$ is still a $k$ dimensional cone if and only if $i \not \in F_k$. In other words, if and only if the corresponding chain is a chain of flats of $\Lambda_{M \backslash i}$. Therefore, we have $$\delta(B(M) \cap U_j) =  B(M \backslash i ) \cap U_j^{\prime},$$ where $U_j^{\prime}$ is a chart of $\T^{n-1}$.  
In addition, $\delta$ contracts a  $k$-dimensional face of $B(M) \cap U_j$ if and only if $i \in F_k$. Thus the image of all contracted faces is exactly $B(M / i) \cap U_j^{\prime} \subset B(M \backslash i) \cap U_j^{\prime}$. 

By the next lemma the codimension one cycle $B(M / i) \cap U_j^{\prime}$ must be the divisor of a tropical rational function $f$ on $B(M \backslash i ) \cap U^{\prime}_j$. Then up to tropical multiplication by a constant (addition) this function  must satisfy $$\Gamma_f(   B(M \backslash i ) \cap U^{\prime}_j  ) \subset B(M) \cap U_j$$  and so it must be the function of the modification $\delta$. 
%
%We can extend this to be a map $\T^n \longrightarrow \T^{n-1}$. Choosing a chart $ \TP^n \supset U \longrightarrow \T^n \longrightarrow \T^{n-1}$ so that $B(M \backslash i) \cap U \subset B(M \backslash i)$ and  $B(M / i) \cap U \subset B(M / i)$ are both dense. 
%A facet,  $P \subset B(M)$ corresponding to a maximal chain of flats $F_1 \subset \dots \subset F_k$ in $\Lambda_M$ gets mapped under $p_i$ to a cone in $\R^{n-1}$ corresponding to a chain of subsets $F_1 \backslash i \subset \dots \subset F_k \backslash i$, this is a chain in the lattice of flats of $\Lambda_{M \backslash i}$. And the facets of $B(M)$ that are contracted by the projection correspond to maximal chains in $\Lambda_M$ with $F_1 = \{i\}$.
%
%Removing from $B(M)$ the facets containing the vector $-e_i$, we obtain a graph of a continuous piecewise linear rational sloped function on $B(M\backslash i)$ that is linear on a fan refinement. By the following lemma the function $f$ can be expressed as a quotient of tropical polynomials and so $\delta$ is an elementary modification. 
\end{proof}

\begin{lemma}\label{lem:WeilCartier}
Let $B(M) \subset \TP^n$ be the Bergman fan of a matroid, and $V = B(M) \cap U_i \subset \T^n$ for some $i \in \{0, \dots, n\}$. If  $D \subset V$  is a codimension one tropical subcycle then there exists a tropical rational function $f \in \K_{\T^n}$ such that $\text{div}_{V}(f) = D$. 
\end{lemma}

\begin{proof}
First suppose $V = \T^n$, and  that $D$ has order of sedentarity $0$, then the statement is equivalent to showing that every codimension one cycle in $\R^n$ is the divisor of a tropical function $f \in \K_{\R^n}$. If $D$ is effective, it is a tropical hypersurface and is given by a tropical polynomial by \cite{FirstSteps}. When $D$ is not effective the following argument is due to an idea of Anders Jensen. 
Let $D^-$ denote the collection of facets of $D$ which have negative weights. %and set $p_f = \sum_{E \in D^-} w_{E}$ where $-w_{E}$ is the weight of $E$ in $D$. 
For a face $E$ in $D^-$, there exists a $v \in \Z^n$ and $a \in \R$ such that   $<x, v>= a$ for all $x \in E$. Define a regular function $h_E: \T^n \longrightarrow \T$, 
 by $$h_E(x) =  \mbox{max}\{0, -w_E( <x, v> - a) \} $$ where $w_E<0$ is the weight of $E$ in $D$. The function $h_E$ is given by the tropical polynomial  = $``ax^{-w_E v} + 1_{\T}", $ and  $\text{div}_{\T^n}(h_E)$ is an affine hyperplane containing $E$ and equipped with positive weight $-w_E $.  
Let $h:\T^n \longrightarrow \T$ be given by $h(x) = \sum_{E \in D^-} h_E(x)$, this corresponds to the tropical product of the tropical polynomials, so $h$ is again a tropical polynomial. Moreover, $D + \text{div}_{\T^n}(h)$ is an effective cycle of order of sedentarity zero in $\T^n$, and thus is the divisor of a tropical polynomial $f$.  By part $(3)$ of  Proposition \ref{StabDiv}, $D = \text{div}_{\T^n}(f-h)$, and the difference $f - h$ is a tropical rational function. 

If $D$ does not have order of sedentarity $0$ then $D$ contains boundary hyperplanes of $\T^n$ corresponding to  $x_i = -\infty$ for a choice of $i$, equipped with an integer weight. This boundary cycle  is the divisor of the Laurent monomial $x_i^w$. Because $D$ is of codimension one, it decomposes into a sum of cycles of order of sedentarity $0$ or $1$ so we are done. 

Now if the fan $V \subset \T^n$ is the closure of  a $k$-dimensional subspace  in $\R^n$ and  $D \subset V$ a  codimension one cycle. 
There is a unique surjective linear projection $\delta:V \longrightarrow \T^k$ with kernel generated by standard basis directions.  The image $\delta(D) \subset \T^k$ is isomorphic to $D$ as an integral polyhedral complex. Moreover equipped with the weights from $D$, $\delta(D)$ is a balanced codimension one cycle in $\T^k$. Therefore, it is the divisor of a tropical rational function $f$ on $\T^k$. 
Let $\tilde{f} $ be the pullback of this function to $\T^n$. It is again a tropical rational function and we have $\text{div}_V(\tilde{f}) = D$.

For the general case, take a linear projection $\delta: V \longrightarrow V^{\prime}$, with kernel generated by $e_i$. Denote the divisor of $\delta$ by $D^{\prime} \subset V^{\prime}$. We may assume by induction that a  codimension one cycle in $V^{\prime}$ (and similarly,  $D^{\prime} \times \R$) is the divisor of a tropical rational function. 
Recall the pushforward and pullback maps defined for an elementary modification in Definition \ref{pushModCon}. Then the cycle $\delta_*D \subset V^{\prime}$ is the divisor of a tropical rational function $f$ and  the cycle $\delta^*\delta_*D$ is the divisor of the pullback  $\tilde{f} = f(\delta)$.  The difference $\delta^*\delta_*D - D$ is a cycle contained in the undergraph of $\delta$ and may be considered as a cycle in $D^{\prime} \times \R$ (for details see Lemma \ref{undergraph} of the next section). So, $\delta^*\delta_*D - D$ is the divisor of some tropical rational function $g$ on $D^{\prime} \times \R$. Moreover, we may choose $g$ so that  $D = \text{div}_V(\tilde{f}-g)$, so the claim is proved. 
\end{proof}

The above lemma shows that the tropical analogues of  Weil divisors and Cartier divisors on the Bergman fan of a matroid are equivalent.  However, effective tropical codimension one cycles are not always given by regular tropical functions. Examples of this appear in \cite{MikICM} and \cite{AllRau} and also in Example  \ref{M0n} at the end of this section. 

We remark that even when $i$ is a loop the above proposition holds, but in a particular sense where the function on $B(M\backslash i)$ producing the modification is the constant function $f = -\infty$. The divisor of such a function is all of $B(M \backslash i)$ which is equal to $B(M/i)$, if $i$ is a loop.   

A \textbf{basis} of a matroid $M = (E, r)$ is a subset $B \subseteq E$ such that $|B| = r(B) = r(E)$. 

\comment{
\begin{corollary}
Given a loopless matroid $M$,  the restriction $M / i $ contains loops if and only if  in $U_j$ there corresponding to the modification  $\delta_j: B(M) \cap U_j \longrightarrow B(M \backslash i) \cap U_j$ is the extension to $\T^n$ of an affine linear function on $\R^n$. 
\end{corollary}}
%
%In this case the fans $B(M)$ and  $B(M\backslash i)$ can be taken to each other by linear maps, and so  are isomorphic as polyhedral complexes.
%
\begin{corollary}\label{fullContraction}
Given a $k$-dimensional Bergman fan $B(M) \subseteq \TP^n$, every contraction $\delta:B(M) \longrightarrow \TP^k$ corresponds to a choice of basis of $M$. 
\end{corollary}

\begin{proof}
Given a basis $B$ of $M$ the deletion $M \backslash B^c$ produces the uniform matroid $U_{k+1, k+1}$ corresponding to $\TP^k$. If we delete along a set which is not the complement of some basis then we decrease the rank of the matroid, meaning at some step we deleted a coloop. This does not correspond to a tropical contraction. 
\end{proof}

From now on the focus will be on matroidal fans and matroidal contraction charts. To simplify the notation we will drop the use of $B(M)$ and just insist that a fan is matroidal, we will only recall the underlying matroid when necessary.

\begin{definition}
We call a fan $V \subset \T^n$ $($or $\R^n)$ matroidal if there exists a matroid $M$ on $n+1$ elements such that $V = B(M) \cap U_i$ for some coordinate chart $U_i = \{x_i \neq -\infty\} \subset \TP^n$, $($or $V = B(M) \cap \R^n)$.
\end{definition}

In the next section we will be concerned only with matroidal fans in $\R^n$. 
For this we make clear the notion of \textbf{open matroidal tropical modifications}.

\begin{definition}
Let $V \subset \R^{n}$ and $V^{\prime} \subset \R^{n-1}$ be matroidal fans. An elementary open matroidal tropical modification is  a modification $$\delta: \tilde{V} \longrightarrow V,$$ where $V \subset \R^n$, and $\tilde{V}\subset \R^{n+1}$ are matroidal and the divisor $D \subset V^{\prime}$ of the modification is also matroidal or empty. 
% such that  there exists a tropical rational function $f:\R^n \longrightarrow \R$ with $\text{div}_V(f) \subset V$ also matroidal and satisfying that the graph $\Gamma_f(V) \subset \tilde{V}$. 
\end{definition}

As mentioned in the introduction, an open elementary matroidal  modification $\delta:\tilde{V}\longrightarrow V$  along a function $f$ should be thought of as a embedding of $V$ with the divisor  $\text{div}_V(f)$ removed.  As before, an open matroidal modification is a composition of elementary matriodal open modifications.

%It was shown by Sturmfels in \cite{SturmPoly} that the \textit{logarithmic limit set} \cite{Berg} of a linear space $ P \subset (\C^{\ast})^{n+1}$ is the Bergman fan of the corresponding matroid $M(P)$.  
Let $V \subset \R^n$ be a $k$-dimensional Bergman fan of a matroid $M$.
Let $\textbf{K}$ be the field of Puiseux series with coefficients in a field $\textbf{k}$ of characteristic $p$. 
We say $V$  is \textbf{realisable} over a field of characteristic $p$ 
if there exists a $k$-plane $\textbf{V} \subset  (\textbf{K}^*)^n$ such that $\text{Trop}(\textbf{V})=V$, (see for example \cite{St7} for definitions of $\text{Trop}$ of an algebraic variety). This is equivalent to the corresponding  matroid $M$ being realisable in characteristic $p$, see \cite{tropGrass}.

\begin{figure}

\includegraphics[scale=0.6]{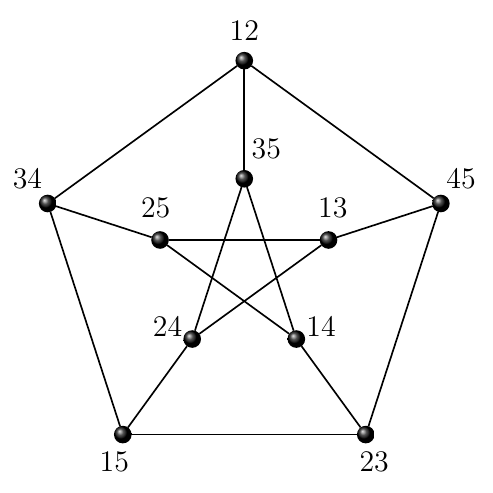}
\put(-95, -20){$\mathcal{M}_{0,5} \subset \R^5$}
\hspace{0.5cm}
\includegraphics[scale=0.6]{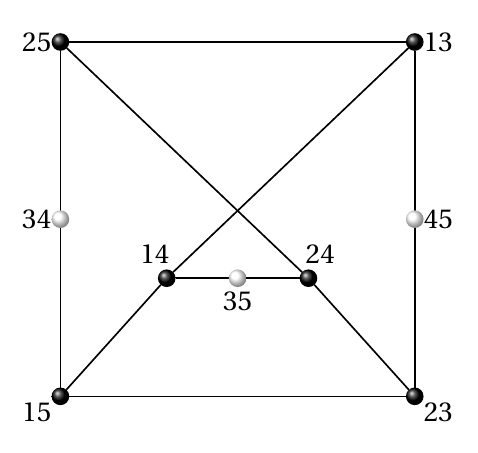}
\put(-85, -20){$V \subset \R^4$}
\vspace{0.5cm}

\includegraphics[scale=0.6]{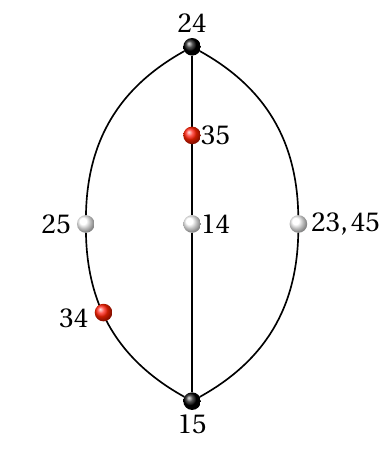}
\put(-75, -20){$V^{\prime} \subset \R^3$}
\hspace{.8cm}
\includegraphics[scale=0.6]{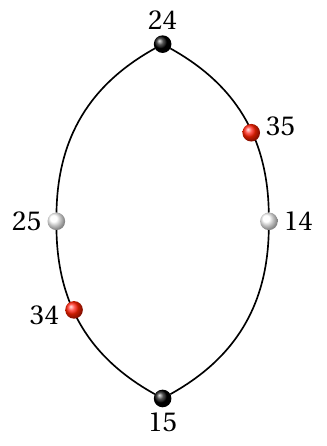}
\put(-55, -20){$\R^2$}
\hspace{0.4cm}
\caption{The link about the origin (or Bergman complex \cite{SturmPoly}) of the sequences of modifications producing $\mathcal{M}^{trop}_{0,5}$. The divisor at each step is marked in white, the $i j$ indicate the cone of $\mathcal{M}^{trop}_{0,5}$ corresponding to a combinatorial type of curve, see \cite{mikMod}, \cite{GathKerbMark} . \label{M05}}

\end{figure}

\vspace{0.3cm}

\begin{proposition}Let   $V_M  \subset \R^k$ be a $k$-dimensional matroidal fan corresponding to the matroid $M$. If $V_M$ is obtained from $\R^k$  by a sequence of elementary regular matroidal modifications 
then $V_M$ is  realisable over a field of characteristic zero. 
\end{proposition}

\begin{proof}
Let $\delta:V \longrightarrow V^{\prime}$ be the first  elementary open regular matroidal modification of a sequence, and let $D \subset V^{\prime}$ be the corresponding divisor.  Then $V^{\prime}$ corresponds to the matroid $M\backslash i$ for some $i \in E$ the ground set of $M$. Also $D $ corresponds to the matroid $M / i$. 
By induction we may assume that $V^{\prime}$  is  realisable over a field of characteristic zero.   Without loss of generality we may also suppose that $D \neq \emptyset$ and  that there is a regular tropical function $f$ with  $\text{div}_{V^{\prime}}(f) = D$ and such that  $\text{div}_{\R^{n-1}}(f) = V_f$ is matroidal. 
Then the  cycle $V_f$ defined by $f$,  is also realisable in the above sense.
  %By part (1) of  Proposition \ref{StabDiv} $D = V^{\prime}.\text{div}_{\R^{n-1}}(f)$. 
%  Since $f$ is regular it is a tropical Laurent polynomial and $\text{div}_{\R^{n-1}}(f)$ is realisable. In fact it may be given by a tropical degree one polynomial, since $V$ is matroidal, therefore  $\text{div}_{\R^{n-1}}(f)$ is also matroidal. 
  By Theorem $4.3$ from \cite{Speyer}, the tropical stable   intersection of two matroidal  is always realisable over the field of Puiseux series with coefficients in $\C$. So the modification of $V^{\prime}$ with center $D$ is realisable by the graph of the function giving $\textbf{D}$ restricted to  $\textbf{V}^{\prime}$, where  $\textbf{D}$  and  $\textbf{V}^{\prime}$ realize $D$ and $V^{\prime}$ respectively. 
\end{proof}

\vspace{0.3cm}
\begin{remark}
Bergman fans obtained via modification by regular functions \textit{do not} correspond to regular matroids, where regular means being realisable over every field. For instance the matroid $U_{2, 4}$ which is not realisable over the field $\mathbb{F}_2$ corresponds to the four valent tropical line in $\TP^3$ which can be obtained by modifications along regular functions. 
\end{remark}

\vspace{0.3cm}
Modification along regular functions is sufficient to ensure realisability, but it is by no means necessary.

\vspace{0.5cm}

\begin{example}\label{M0n}
The embedding of the moduli space of tropical rational curves with $5$ marked points, $\mathcal{M}^{trop}_{0,5}$
into $\R^5$  (see \cite{mikMod}, \cite{tropGrass}, \cite{GathKerbMark} ) is the first example of a realisable fan not obtained by a sequence of modifications along regular functions. The rays of a fine polyhedral subdivision of $\mathcal{M}^{trop}_{0,5}$ may be labelled by distinct pairs $\{i, j\} \subset \{1, \dots, 5\}$ as in Figure  \ref{M05}. See \cite{mikMod}, \cite{ArdKli} for more details.  It was shown in \cite{ArdKli} that $\mathcal{M}^{trop}_{0,n}$ corresponds to the Bergman fan of the complete graphical matroid $K_{n-1}$. Tropical contractions of $\mathcal{M}^{trop}_{0,n}$ correspond to the deletion of an edge of $K_{n-1}$. 
  Therefore, the very first elementary tropical  contraction of $\mathcal{M}^{trop}_{0,5}$ is unique by the symmetry of $K_4$.
The link of singularity of the fans obtained by a series of elementary contractions starting from $\mathcal{M}^{trop}_{0,5}$ and finishing at $\R^2$ are drawn in Figure \ref{M05}. The divisors of each modification marked in white.  It will be shown in Section \ref{sec:surfaces}  that the corresponding divisor of this contraction cannot be the divisor of a regular function on $\R^4$ restricted to $V$ by showing that its tropical self intersection is not effective, which would contradict Corollary \ref{regEff}. 
\end{example}

For open matroidal  modifications we have the following proposition regarding the pullback and pushforward cycle maps  given in Definition \ref{pushModCon}.

\begin{proposition}\label{prop:homo}
Given an open matroidal  modification $\delta:\tilde{V} \longrightarrow V$ the maps $\delta^{\ast}:Z_{k}(V) \longrightarrow Z_{k}(\tilde{V})$ and $\delta_{\ast}:Z_{k}(\tilde{V}) \longrightarrow Z_{k}(V)$ are group homomorphisms for all $k$, and $\delta_{\ast}\delta^{\ast} = id$. 
\end{proposition}

\begin{figure}
\includegraphics[scale= 0.35]{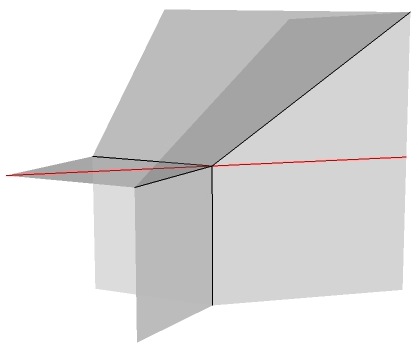}
\hspace{1.5cm}
\includegraphics[scale=0.35]{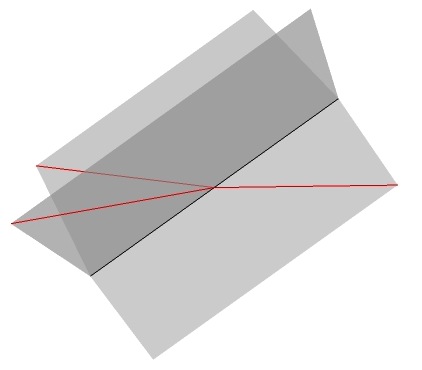}
\put(-200 ,10){$V_1$}
\put(-200, 70){\tiny{$\text{div}_{f_1}(V_1)$}}
\put(-20, 50){\tiny{$\text{div}_{g_2}(V_2) $}}
\put(-50 ,10){$V_2$}
\caption{The two cycles $V_1$, $V_2 \subset \R^3$ from Example \ref{ex:twomodifications}, the divisors $\text{div}_{f_1}(V_1) \subset V_1$, $\text{div}_{g_2}(V_2) \subset V_2$ are drawn in red in each case.    \label{fig:twomodifications}}
\end{figure}

\begin{proof}
It was already mentioned that the pushforward and pullback maps are group homomorphisms when the modification is elementary. Therefore, we must only show that the maps $\delta_{\ast}, \delta^{\ast}$ are well defined when we compose  open elementary modifications. 
Suppose $\delta: \tilde{V} \longrightarrow V$ is the composition of two open matroidal modifications. Set $\delta_2: \tilde{V} 
\longrightarrow V_2$ and      $ \tilde{\delta}_1: V_2 \longrightarrow V$, so that $\tilde{\delta}_1\delta_2 =\delta$, and denote the other sequence of modifications by $\delta_1: \tilde{V} \longrightarrow V_1$ and $\tilde{\delta}_2: V_1 \longrightarrow V$, so that $\tilde{\delta}_2\delta_1 = \delta$, (see Example \ref{ex:twomodifications} for a case when the fans $V_1$ and $V_2$ differ). Without loss of generality we may suppose the kernels of $\delta_1, \delta_2: \R^{n+2} \longrightarrow \R^{n+1}$ are generated by $e_{n+1}$ and $e_{n+2}$, respectively. Then the maps $\tilde{\delta}_1$ and $\tilde{\delta}_2$ are linear projections $\R^{n+1} \longrightarrow \R^n$, with kernels $e_{n+1}$ and $e_{n+2}$ respectively. 
 
For the pushforwards, the sets satisfy, $\tilde{\delta}_i\delta_j(A) = \tilde{\delta}_j\delta_i(A)$, since the $\delta_i$'s and $\tilde{\delta}_i$'s  are orthogonal projections. Let $\mathcal{C}$ denote the closure of the collection of facets of $C$ contracted by both $\delta_1, \delta_2$. If a facet $F$ of $A$ is outside of $\mathcal{C}$ then its contribution to the weight of $\delta(F) \subset \delta_{\ast}A$ is the same if we permute the order of contractions. So assume $F \subset \mathcal{C}$, then the lattice index may be  rewritten as, $[\delta_{i\ast}\Lambda_F : \Lambda_{\delta_i(F)}] = [\Z^n: \Lambda_F + \Lambda_i^{\perp}]$
and  $F$ contributes a weight of, 
$$w_A(F)[\Z^n: \Lambda_F + \Lambda_i^{\perp}][\Z^n: \Lambda_{\delta_1(F)}+ \Lambda^{\perp}_j] = w_A(F)[\Z^n:\Lambda_F + \Lambda_i^{\perp} \cap \Lambda_j^{\perp} ],$$ to $\delta(F)$. Which is independent of the order of contractions. 
 
For the pullbacks,  take a cycle $A$ in $V$ and let $\Gamma(A) \subset \R^{n+2}$ denote the graph of  $A$ along either pair of functions yielding the modification. Although the pairs of functions may differ, (see Example \ref{ex:twomodifications}), the resulting graphs must be the same. Let $\tilde{A}$ denote the pullback of $A$ along the composition $\tilde{\delta}_2\delta_1$ and $\tilde{A}^{\prime}$ denote the pullback of $A$ along the composition $\tilde{\delta}_1\delta_2$.  %and any facets of  $\tilde{A} \backslash \Gamma(A)$ and $\tilde{A}^{\prime} \backslash \Gamma(A)$ not containing the direction $-e_{n+1}$ and $-e_{n+2}$ remain facets under  the images of the projection maps $\delta_2$ and $\delta^{\prime}_1$. 
Since $\tilde{A}$ and $\tilde{A}^{\prime}$ are modifications of cycles in $V$  the restriction of the linear projections $\delta_1$ and $\delta_2$ to $\tilde{A}$ and $\tilde{A}^{\prime}$  are either one to one or send a half line to a point. 
 
If  $E_A$ is an unbalanced codimension one face of $\Gamma(A)$, then it is unbalanced only in the $e_{n+1}$ and $e_{n+2}$ directions. 
First, if $\tilde{V}$ contains one or both of the faces:
$$\{ x -te_{n+1} \ | \ x \in E_A \text{ and } t \in \R_{\geq 0 } \}, \ \{ x -te_{n+2} \ | \ x \in E_A \text{ and } t \in \R_{\geq 0 } \}, $$ then these are the only facets of $\tilde{A}$ adjacent to $E_A$ and not contained in $\Gamma(A)$, and similarly for $\tilde{A}^{\prime}$. 
The balancing condition at $E_A$ guarantees that the weights are the same, (remark that if $\Gamma(A)$ is  already balanced in one of these directions then we do not need to add the corresponding facet). 

For an unbalanced codimension one face  $E_A$ of $\Gamma(A)$  suppose the above faces do not exist. Then there is a single facet $F_{\tilde{A}}$ of $\tilde{A}$ adjacent to $E_A$ and not in $\Gamma(A)$. Otherwise the projections $\delta_1$ and $\delta_2$ restricted to $\tilde{A}$ would have a finite fiber of size at least two.   The same holds for  $\tilde{A}^{\prime}$, whose single face satisfying these conditions we call $F_{\tilde{A}^{\prime}}$. Now $\tilde{A} - \tilde{A}^{\prime}$ must be balanced at $E_A$ and so the faces $F_{\tilde{A}}$ and $F_{\tilde{A}^{\prime}}$ are the same and equipped with the same weights. 

In this case  there may be codimension one faces of $F_{\tilde{A}}$ at which there are other facets of $\tilde{A}$ adjacent. This occurs when the divisor $D_1\subset V_1$ of  the modification $\delta_1$  is contained in the undergraph of the modification $\tilde{\delta}_2$ and  $\tilde{\delta}^*_2A$ and in addition intersects $D_1 \subset V_1$ in some codimension one face. Call the resulting codimension one face $G_A$ of $F_{\tilde{A}} \subset \tilde{A}$. Then $G_A$ is contained  in the skeleton of $\tilde{V}$ and it is also a face of $F_{\tilde{A}}^{\prime} \subset \tilde{A}^{\prime}$.
If the cycles are unbalanced at  $G_A$ the other facets adjacent to it in $\tilde{A}$ and $\tilde{A}^{\prime}$ must be: 
$$\{ x -te_{n+1} \ | \ x \in G_A \text{ and } t \in \R_{\geq 0 } \}, \ \{ x -te_{n+2} \ | \ x \in G_A \text{ and } t \in \R_{\geq 0 } \},$$ otherwise the projections $\delta_1$, $\delta_2$ would have a finite fiber of size greater than one.   Again, by the balancing condition the weights of these faces in $\tilde{A}$ and $\tilde{A}^{\prime}$ agree. 
\end{proof}

The following example shows a  composition of open matroidal modifications for which the intermediary fans $V_1, V_2$ appearing in the proof above are not the same. 

\begin{example}\label{ex:twomodifications}
Consider the fan $V \subset \R^4$ obtained from $\R^2$  via two elementary open modifications,  $\delta_1,  \delta_2$. The first modification is along the function $f_2(x, y) = \max \{x, y, 0\}$ and yields the  cycle $V_1 \subset \R^3$ shown on the left of Figure \ref{fig:twomodifications}. The next modification is taken along the function  $f_1:\R^3 \longrightarrow \R$ given by $$f_1(x, y, z) = \max \{x, y\} + \max \{z, 0\} - \max \{x, y , z, 0\}. $$

It may be verified that the following different sequence of modifications yields the same fan, $V \subset \R^4$,  after a change of coordinates.  If one first modifies $\R^2$ along the function $g_1(x, y) = \max \{x, y\}$, to obtain a cycle $C_2 \subset \R^3$, see the right hand side of Figure \ref{fig:twomodifications}.  Next,  modify $V_2$ along the function $g_2 : \R^3 \longrightarrow \R$, given by $g_2(x, y, z) = \max \{ z, 0\}$.  Notice on the one hand $V$ is produced by a composition of two elementary regular modifications, and on the other by an elementary regular modification composed with an elementary modification along a rational function. 

\comment{
All of the above fans are tropicalisations of planes in  $(\C^*)^N$. To begin with $\R^2$ is simply $(\C^*)^2$, which is the complement of $3$ lines in $\CP^2$. The fan $V_1 \subset \R^3$ is the tropicalisation of a plane in $(\C^*)^3$ which defines a line arrangement on the left of Figure \ref{fig:twomodlines}, and the fan  $V_2 \subset \R^3$ is the tropicalisation of a plane in $(\C^*)^3$ with underlying line arrangement on the right of Figure \ref{fig:twomodlines}. }
\end{example}

\end{subsection}

\end{section}

\begin{section}{Intersections in matroidal fans}\label{Intersections}
%:Intersections
In this section we intersect tropical subcycles of an open matroidal fan $V \subset \R^n$, so throughout we restrict our attention  to open matroidal tropical modifications.  Set $\dim(V) = k, \dim(A) = m_1, \dim(B) = m_2$, and the expected dimension of intersection of $A$ and $B$ to be $m = m_1 + m_2 - k$. Also for any complex $C$ whose support is contained in $V$, let $C^{(s)}$ denote the $s$-dimensional skeleton of $C$ with respect to the refinement induced by the inclusion to $V$.

\begin{definition}\label{defTrans}
Let  $V \subset \R^n$ be a matroidal fan and $A, B \subset V$ be subcycles. 

\begin{enumerate}
\item{$A \cap B$ is proper in $V$ if $A \cap B$ is of pure dimension $m$ or is empty.}

\item{$A \cap B$ is weakly transverse in $V$ if every facet of $(A \cap B)^{(m)}$ is in the interior of a facet of $V$.}

\item{$A \cap B$ is transverse in $V$ if it is proper, weakly transverse and every facet of $A \cap B$ comes from facets of $A$ and $B$ intersecting transversely. }

\end{enumerate}

\end{definition}

\begin{example}\label{negCurve}
The standard hyperplane $P \subset \R^3$ was shown in Figure \ref{Plane}, it is obtained by modifying $\R^2$ along the standard tropical line.  Let $A$ be the sub-cycle parameterized by $(t, t,  0)$ and $B$ be the union of the positive span of the  rays $(0, 1, 1), \ (1-d, -d, 0), \ (d-1, d -1, -1)$, see Figure \ref{intEx}.

\begin{figure}
i)\includegraphics[scale=0.35]{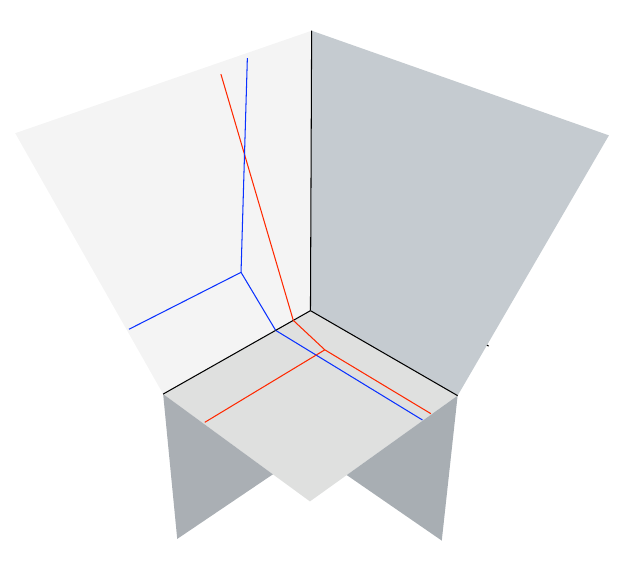}
ii)\includegraphics[scale=0.38]{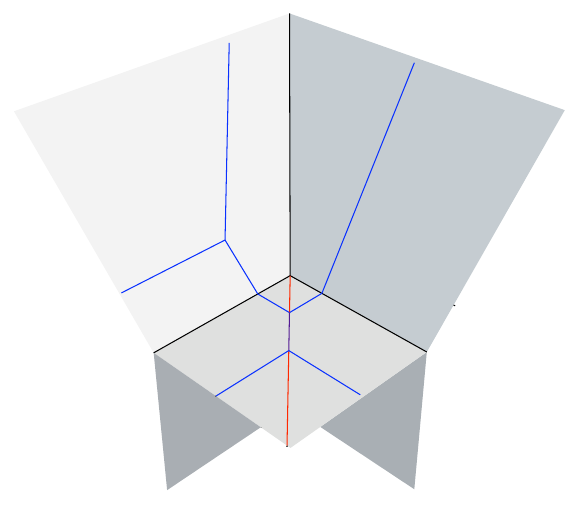}
iii)\includegraphics[scale=0.31]{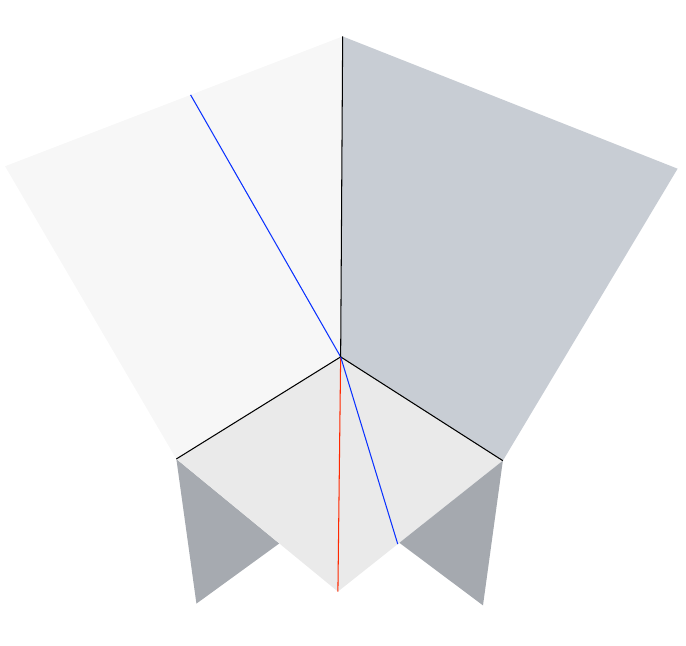}
\caption{Cycles in the standard hyperplane in $\R^3$. i) Transverse intersection ii) Weakly transfer intersection iii) Neither}
\label{figureTrans}
\end{figure}

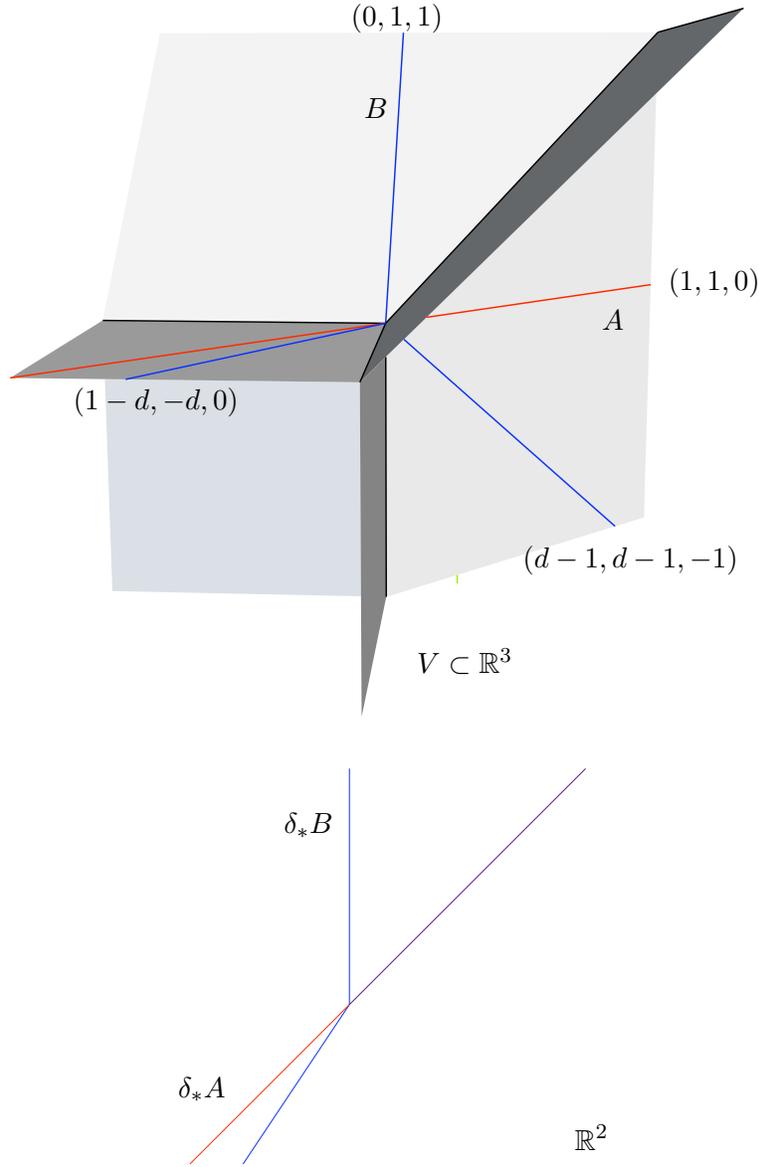
\begin{figure}[ht]
\includegraphics[scale=1.1]{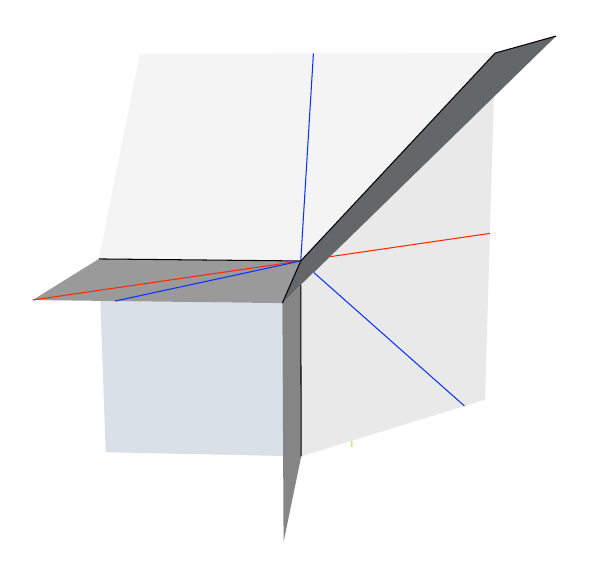}
\put(-55, 175){$(1,1,0)$}
\put(-80, 160){$A$}
\put(-110, 70){$(d-1,d-1,-1)$}
\put(-280, 130){$(1-d, -d, 0)$}
\put(-175, 275){$(0,1,1)$}
\put(-170, 240){$B$}
\put(-150, 30){$V \subset \R^3$}

\includegraphics[scale=0.5]{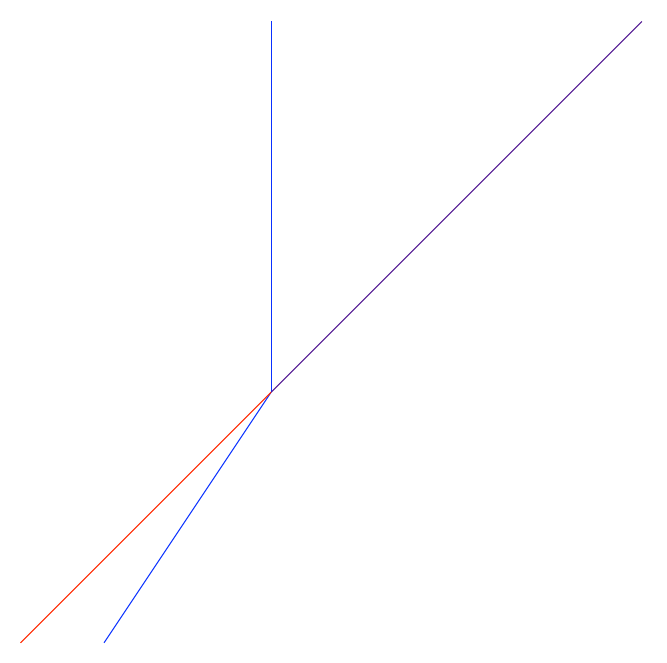}
\put(-120, 130){$\delta_{\ast}B$}
\put(-160, 30){$\delta_{\ast}A$}
\put(-10, 10){$\R^2$}
\caption{Tropical cycles in the standard hyperplane in $\R^3$ along with the image under the contraction to $\R^2$. \label{intEx}}

\end{figure}

The curves $A$ and $B$ intersect only at the vertex $p$ of the fan. This intersection is proper but not weakly transverse. Moreover both cycles are rigid in $P$, meaning they cannot be moved in $P$ by a translation. Consider the  contraction $\delta:P \longrightarrow \R^2$, given by projecting in the $e_3$ direction. %with divisor $D$ the standard tropical line in $\R^2$. 
 Set $\Delta_A =   \delta^{\ast}\delta_{\ast}A - A$ and $\Delta_B =  \delta^{\ast}\delta_{\ast}B - B$. An intersection product should of course be distributive, so we ought to have, 
\begin{align*}
A.B = & \ (  \delta^{\ast}\delta_{\ast}A - \Delta_A ).(\delta^{\ast}\delta_{\ast}B - \Delta_B) \\
= & \  \delta^{\ast}\delta_{\ast}A.\delta^{\ast}\delta_{\ast}B  - \delta^{\ast}\delta_{\ast}A. \Delta_B-   \Delta_A.\delta^{\ast}\delta_{\ast}B + \Delta_A.\Delta_B
\end{align*}
Now, the cycles $ \delta^{\ast}\delta_{\ast}A, \delta^{\ast}\delta_{\ast}B$ are free to move in $P$ in the same way that $\delta_{\ast}A, \delta_{\ast}B$ are free to move in $\R^2$. By translating $\delta^{\ast}\delta_{\ast}A, \delta^{\ast}\delta_{\ast}B$ until they intersect transversally and then translating back we can associate the weight, $$w_{\delta^{\ast}\delta_{\ast}A.\delta^{\ast}\delta_{\ast}B}(p) = 1 = w_{\delta_{\ast}A.\delta_{\ast}B}(\delta(p)).$$
The cycles $\Delta_A, \Delta_B$ are contained in the undergraph of the modification,  see Figure \ref{DtimesR},  and are free to move in this direction.  Also the cycle $\delta^{\ast}\delta_{\ast}A$ restricted to the undergraph is just $\mbox{div}_A(f) \times \R$, and similarly for $\delta^{\ast}\delta_{\ast}B$ .

Now the cycles $\Delta_A, \Delta_B$ may be moved by a translation into a single facet of $P$, see Figure \ref{DtimesR}. 
We can calculate $$w_{\delta^{\ast}\delta_{\ast}A.\Delta_B }(p)= w_{\Delta_A.\delta^{\ast}\delta_{\ast}B}(p)= 0,$$ and 
$$w_{\Delta_A.\Delta_B}(p) = 1-d.$$
Combining all of these we obtain:
\begin{align*}
w_{A.B}(p) = &w_{\delta^{\ast}\delta_{\ast}A.\delta^{\ast}\delta_{\ast}B}(p) -w_{\delta^{\ast}\delta_{\ast}A.\Delta_B }(p)\\ 
& - w_{\Delta_A.\delta^{\ast}\delta_{\ast}B}(p) + w_{\Delta_A.\Delta_B}(p) = -d+2.
\end{align*}

\begin{figure}[ht]
\includegraphics[scale=0.9]{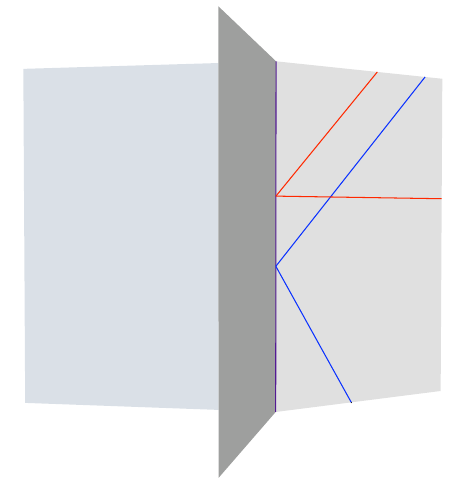}
\put(-55, 115){$\Delta_A - te_3$}
\put(-60, 65){$\Delta_B - t^{\prime}e_3$}
\put(-150, 10){$D \times \R$}

\includegraphics[scale=1]{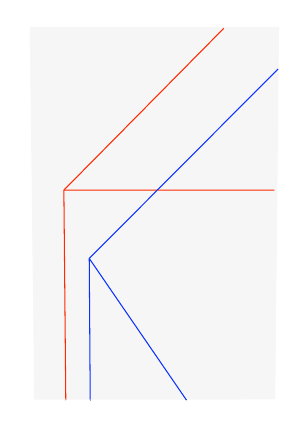}
\put(-140, 170){$\Delta_A - (t_1, t_1, t_2)$}
\put(-90, 75){$\Delta_B - (t_1^{\prime},t_1^{\prime},t_2^{\prime})$}
\put(-70, 105){\small{$1-d$}}
\caption{The second and then third translations of the cycles $\Delta_A, \Delta_B$ from Example \ref{negCurve}.} 
\label{DtimesR}
\end{figure}

\end{example}

Our aim is to obtain a general procedure to split cycles contained in a matroidal fan $V$ in a way so that they may be intersected. 
To do this  we first need some technical definitions and lemmas. 
\begin{definition}
Given $\delta: V \longrightarrow V^{\prime}$ an elementary open matroidal  modification, let $f$ denote the corresponding tropical rational function and $D$ its divisor.  Let $A \subset V$ a be cycle, then denote:

\begin{enumerate}
\item $\Delta_A = \delta^{\ast} \delta_{\ast}A - A$.
  
\item $D_A = \mbox{div}_{\delta_{\ast}A}(f) \times \R \subset D\times \R$. 
\end{enumerate}  

\end{definition}

\begin{lemma}\label{deltadhomo}
Given $\delta: V \longrightarrow V^{\prime}$ an elementary open matroidal  modification, let $f$ denote the corresponding tropical rational function and $D$ its divisor.
Let $A, B$ be subcycles of $V$ then, 
\begin{enumerate}
\item $\Delta_{A+B} = \Delta_A + \Delta_B$
\item $D_{A+B} = D_A + D_B$.
\end{enumerate}

\end{lemma}

\begin{proof}
The first statement is clear since $\delta^{\ast}, \delta_{\ast}$ are homomorphisms, and the second follows from $\mbox{div}_{A+B}(f)= \mbox{div}_A(f) + \mbox{div}_B(f)$.
\end{proof}

\comment{For an elementary tropical modification $\delta: \tilde{V} \longrightarrow V$ given by a  function $f$ recall in Section \ref{sec:mods} we defined its undergraph $\mathcal{U}(\Gamma_f(V))$. For an open modification we may define the undergraph  in an analogous way. Note that in this case, $\mathcal{U}(\Gamma_f(V)) \subset \R^{n+1}$. }

\begin{lemma}\label{undergraph}
Let $\delta: V \longrightarrow V^{\prime}$ be  an  elementary open matroidal modification along the rational  function $f$ and having divisor $D$.   If a cycle $A \subset V$ is in $\mbox{Ker} \ \delta_{\ast}$, then it is contained in the closure of the undergraph $\mathcal{U}(\Gamma_f(V^{\prime}))$.   In particular, it is also a subcycle of $D \times \R$ where $\R$ is the affine space spanned by the kernel of $\delta$. 
\end{lemma}
 
 \begin{proof}
Away from the divisor $D \subset V$ the map $\delta$ is one to one thus no cancellation of facets can occur in $\delta_{\ast}A$ outside of $D$.  So $\delta(A)$ must be contained in $D$ which implies the lemma.  
 \end{proof}
 
A quick check shows  that $\Delta_A$ is in the kernel of $\delta_{\ast}$, since $ \delta_{\ast} \delta^{\ast} = \mbox{id}$. 
Therefore, for an elementary open modification of matroidal fans $\delta: V \longrightarrow V^{\prime}$ and any cycle  $A \subset V$ we have $\Delta_A, D_A \subset D \times \R$, where $D \subset V^{\prime}$ is the divisor of the modification. Using this we define an intersection product on $V$ in terms of a product on $V^{\prime}$ and $D \times \R$. 
%Fixing an elementary open modification $\delta:V \longrightarrow V^{\prime}$ we define the intersection product  as follows:

\begin{definition}\label{stableIntMat}
Given cycles $A, B \subset V \subset \R^n$ and an elementary open matroidal modification $\delta:V \longrightarrow V^{\prime}$ with associated divisor $D$, define, 
$$A.B = \delta^{\ast}(\delta_{\ast}A.\delta_{\ast}B) +C_{A.B}$$ with 
 $$C_{A.B} =  \Delta_{A}.\Delta_{B} - \Delta_{A}.D_B - D_A.\Delta_{B},$$
 where these products are calculated in the matroidal fan $D \times \R \subset \R^n$. 
\end{definition}

The above definition gives the  product of two cycles $A, B$ in $V$ as a sum of products of cycles in fans $V^{\prime}$ and $D \times \R$, one of which is of lower codimension, and the other containing the linear space spanned by the kernel of $\delta$.  Continuing to apply this procedure to $V^{\prime}$ and $D$ we continue to decrease the codimension or increase the dimension of the affine linear space contained in the fan and we can eventually reduce the intersection product in $V$ to a sum of pullbacks of stable intersections in $\R^k$, where $k$ is the dimension of $V$.   A priori this definition depends on the choice of all contraction charts. Before showing the above definition is independent of the chosen charts in Proposition \ref{welldef} we state some properties of the intersection product as defined relative to a  fixed  collection of  open matroidal contractions. 

\begin{lemma}\label{pushpullLemma}
Suppose $\delta:V \longrightarrow V^{\prime}$ is an elementary open modification of matroidal fans and $A, B$ are cycles in $V^{\prime}$. The intersection product in $V$ from Definition \ref{stableIntMat} calculated via the modification $\delta$ satisfies
$$\delta^{\ast}A.\delta^{\ast}B = \delta^{\ast}(A.B).$$
\end{lemma}

\begin{proof}
In this case $\Delta_A, \Delta_B = 0$ so the term $C_{A, B}$ from Definition \ref{stableIntMat} is also $0$.  
\end{proof}

\begin{corollary}
Suppose the matriodal fan $V \subset \R^n$ is a $k$-dimensional subspace of $\R^n$, and let  $\delta:V \longrightarrow \R^k$ be an open matroidal contraction. For subcycles $A, B$ in  $V$ we have, 
$$A.B = \delta^*(\delta_{\ast}A. \delta_{\ast}B).$$
\end{corollary}

\begin{proposition}\label{propMatInt}
Let $V \subset \R^n$ be a matroidal fan and $A, B, C$ be subcycles of $V$. 
%Suppose $\delta:V \longrightarrow V^{\prime}$ is an elementary open modification of matroidal fans and $A, B, C$ are cycles in $V$. 
Then the intersection product  given in Definition \ref{stableIntMat} relative to any choice of contraction charts satisfies the following:
%The intersection product for cycles $A, B, C \subset V$ from Definition \ref{stableIntMat} defined relative to fixed collection of contraction charts satisfies the following: 
\begin{enumerate}
\item{$A.B$ is  a balanced cycle contained in $V$}
\item{$A.C = C.A$}
\item{$A_1.(A_2 +A_3) = A_1.A_2 +A_1.A_3$} \label{distributivity}
\item{$A_1.(A_2.A_3)= (A_1.A_2).A_3$} \label{associativity}
\item{$\mbox{div}_A(g) = \mbox{div}_V(g).A$}
\end{enumerate}
\end{proposition}

\begin{proof}
The above properties all  follow by  induction. % on the codimension of $V \subset \R^n$ and on the maximum codimension of the affine linear space contained in $V$. 
The base case being $V = \R^k$, where all of the above properties are satisfied.
Suppose we have chosen, $\delta:V \longrightarrow V^{\prime}$ as the first elementary open matroidal contraction, and let its divisor be $D \subset V^{\prime}$.  We may assume all of the properties stated above hold for intersections in $V^{\prime}$ and $D \times \R$. 

For $(1)$, the weighted balanced complex,  $A.B$ is the sum of $\delta^{\ast}(\delta_{\ast}A.\delta_{\ast}B)$ and $C_{A.B}$ which are both balanced by the induction assumption, so it is balanced.  
Commutativity also follows immediately by induction.  By Lemma \ref{deltadhomo} and  distributivity for products  in $V^{\prime}$ and $D \times \R$, we get distributivity in $V$.

For associativity, first notice that 
\begin{align} 
\Delta_{A_i.A_j} & =  \Delta_{A_i}.D_{A_j} +  D_{A_i}.\Delta_{A_j} - \Delta_{A_i}.\Delta_{A_j} & \\  
 D_{A_i.A_j} & =  D_{A_i}.D_{A_j}. & 
\end{align}
The first line follows from the definition of $\Delta_{A_i.A_j}$.  The statement $(3)$  follows from Lemma \ref{linearity} which follows this proposition.
 Then, 
 $$A_1.(A_2.A_3) = \delta^{\ast}(\delta_{\ast}A_1.(\delta_{\ast}A_2.\delta_{\ast}A_3)) - \Delta_{A_1}.D_{A_2.A_3} - D_{A_1}.\Delta_{A_2.A_3} + \Delta_{A_1}.\Delta_{A_2.A_3}$$
Assuming associativity in $V$ and $D \times \R$ and using commutativity we can remove brackets and write:
 \begin{align*}
 A_1.(A_2.A_3) = & \delta^{\ast}(\delta_{\ast}A_1.\delta_{\ast}A_2.\delta_{\ast}A_3)  +  \\
 &  \sum_{\stackrel{1\leq i<j \leq 3}{k \neq i, j}} \Delta_{A_i}.\Delta_{A_j}.D_{A_k} -   D_{A_i}.D_{A_j}.\Delta_{A_k}  - \Delta_{A_1}.\Delta_{A_2}.\Delta_{A_3}
 \end{align*}
 Regrouping terms and using $(2)$ and  $(3)$ we get, 
 \begin{align*}
 A_1.(A_2.A_3) & =   \delta^{\ast}((\delta_{\ast}A_1.\delta_{\ast}A_2).\delta_{\ast}A_3) - \Delta_{A_1.A_2}.D_{A_3} - D_{A_1.A_2}.\Delta_{A_3} + \Delta_{A_1.A_2}.\Delta_{A_3} \\
 &  =  (A_1.A_2).A_3.
 \end{align*}
 
Lastly,  given a divisor $D = \mbox{div}_V(g)$ we may write it as $\delta^{\ast}\delta_{\ast}D - \Delta_D$. Then $\tilde{g}(x) = g(\delta(x))$, is the function of the divisor $\delta^{\ast}\delta_{\ast}D$ where $f$ is the function of the modification $\delta$. So $\tilde{g} - g$ gives $\Delta_D$ by part $3$ of Proposition \ref{StabDiv}. The result follows by distributivity and by applying the induction hypothesis to both parts. 
\end{proof}

We require a final lemma before proving that the product is independent of the choice of contractions. 
\begin{lemma}\label{linearity}
Let $V \subset \R^n$ be a matroidal fan and  $A, B$ be subcycles of $ V$, set $$\tilde{A} = A \times \R,  \quad \tilde{B} = B \times \R, \quad \text{and} \quad \tilde{V} = V \times \R.$$ Then, we may choose contraction charts  so that  by Definition \ref{stableIntMat} we have $$\tilde{A}.\tilde{B} = A.B \times \R \subset \tilde{V}.$$
\end{lemma}

\begin{proof}
The above statement holds for stable intersections in $\R^n$ and $\R^{n+1}$. If $V$ corresponds to a matroid $M$ on $E$ then $\tilde{V}$ corresponds to a matroid $\tilde{M}$ on $E \cup e$ with bases $B \cup e$ for every base $B$ of $M$, in other words we have added a coloop $e$  to the matroid $M$. Given an elementary open  modification of matroidal fans,  $\delta: V \longrightarrow V^{\prime}$ with divisor $D$ we have a corresponding elementary open modification  $\tilde{\delta}: \tilde{V} \longrightarrow \tilde{V}^{\prime}$ with divisor $\tilde{D} = D\times \R$ and $\tilde{V}^{\prime} = V^{\prime} \times \R$.  In order to define the product in $A.B$, a collection of contractions  are fixed. To intersect $\tilde{A}, \tilde{B}$ in $\tilde{V}$, simply choose the corresponding collection of contractions of $\tilde{V}$.  Applying Definition \ref{stableIntMat} we obtain the lemma by induction. 
\end{proof}

\begin{theorem}\label{welldef}
The intersection product from Definition \ref{stableIntMat} is independent of the choice of open matroidal contractions. 
\end{theorem}

\begin{proof}
Fix a matroidal fan $V \subset \R^n$ and subcycles $A, B$ of $V$. We may assume by induction that the product is well-defined on $D \times \R$ and $V^{\prime}$ where $\delta:V \longrightarrow V^{\prime}$ is any 
elementary open matroidal modification and $D$ is its associated divisor. 

By Corollary \ref{fullContraction} any two open matroidal  contractions $\delta, \delta^{\prime}:V \longrightarrow \R^k$ can be related by a series of basis exchanges. So it suffices to check two things: that we may transpose the order of any two elementary open contractions to $\R^k$ and obtain the same intersection cycle and that if $\delta:V \longrightarrow V^{\prime}$ is the composition of any two elementary open matroidal modifications, we may permute the order of the elementary contractions and obtain the same product.
 In other words we must show that the definition does not depend on the paths taken in the following two diagrams:

$$ \xymatrix{ &  V \ar[ld]_{\delta_1} \ar[rd]^{\delta_2}  &  \\ V_1 \ar[rd]_{\tilde{\delta}_2}   &   & V_2  \ar[ld]^{\tilde{\delta}_1} \\  & V^{\prime}  & }  \ \ \ \ \ \ \ \ \ \ \ \ \ \ \  
 \xymatrix{ &  V \ar[ldd]_{\delta_1} \ar[rdd]^{\delta_2} &    \\ \\ \R^k \ar@{.>}[rr]_{e_2 \mapsto e_1}  &   & \R^k }  $$
We will start by showing the latter, let $\delta_1,\delta_2: V \longrightarrow \R^k$ be two elementary open matroidal contractions. Then $V$ is of codimension one in $\R^{k+1}$ and thus corresponds to a corank one matriod $M$. Suppose without loss of generality that the open contractions $\delta_i$ correspond to the deletion of the element $i$ from the corresponding matroid, Then we may assume that $i= 1, 2$ are not coloops of $M$. 
If we exchange any two non coloop elements $i$ and $j$ of a corank one matroid $M$ we obtain a matroid isomorphism.  Also,  restricting the matroid $M$ to $i$ or $j$ produces isomorphic matroids $M / i \cong M / j$. 
Therefore the divisors $D_i, D_j \subset \R^k$ of the corresponding elementary open matroidal  modifications $\delta_i, \delta_j: V \longrightarrow \R^k$ can be identified as well as the functions $f_i, f_j$ on $\R^k$.

First we will construct cycles $\tilde{A}, \tilde{B} \subset V$ such that $\delta_{1\ast} \tilde{A} = \delta_{2\ast} \tilde{A} \subset \R^k$ and  $\delta_i^{\ast} \delta_{i\ast} \tilde{A} = \tilde{A}$ for $i = 1, 2$ and similarly for $\tilde{B}$. Then by the above remarks concerning the two modifications the definition of the product $\tilde{A}.\tilde{B} = \delta^{\ast}_i (\delta_{i\ast} \tilde{A}. \delta_{i\ast} \tilde{B})$ does not depend on the choice of $i = 1, 2$. 

To construct $\tilde{A}$ and $\tilde{B}$, let $\mathcal{C} \subset V$ denote the union of all faces of $V$ that are not generated by the vectors $v_1, v_2$, where $v_i$ generates the kernel of $\delta_i$. Let $\tilde{A} = \delta^{\ast}_i\delta_{i\ast}\delta^{\ast}_j\delta_{j\ast}A$ and similarly for $\tilde{B}$. The cycle $\tilde{A}$ (respectively, $\tilde{B}$) is well-defined independent of the order of $\delta_i, \delta_j$ since it is obtained from $A \cap \mathcal{C}$ (respectively, $B \cap \mathcal{C}$) by adding uniquely weighted facets to all codimension one faces $E$ of $A$  (respectively, $B$), parallel only to the cones spanned by $E$ and $v_i$ for $i =1, 2$,  so that the result satisfies the balancing condition.  Similarly, in $\R^k$ we have $\delta_{i\ast} \tilde{A} = \delta_{j\ast} \tilde{A}$ and analogously for $\tilde{B}$, since the weighted complexes  $\delta_{i\ast}\tilde{A} \cap \delta(\mathcal{C})$ are equal for $i = 1, 2$ and balanced in all but the $\delta_i(v_j)$ direction where $j=1,2$ and $i\neq j$. Adding the necessary uniquely weighted facets to the codimension one faces of this complex in the $\delta_i(v_j)$ direction gives  $\delta_{i\ast}\tilde{A}$ for $i = 1,2$ and similarly for $\delta_{i\ast}\tilde{B}$. Also by  construction we have $\delta^{\ast}_i\delta_{i\ast} \tilde{A} = \tilde{A}$, and similarly for $B$. 

For $i = 1, 2$, define $\Delta^i_A = \delta^{\ast}_i\delta_{i\ast}A - A$ and $D^i_A =  \mbox{div}_A(f_i) \times \R  \subset D_i \times \R $ and similarly for $B$.
Assume first that $A= \tilde{A} - \Delta^1_A -  \Delta^2_A$, and analogously for $B$.  It follows that $\delta^{\ast}_j\delta_{j\ast}\Delta^i_A = \Delta^i_A$, and similarly for $B$. Then we obtain, $$A.B =  \delta_i^{\ast}(\delta_{i\ast}(\tilde{A} - \Delta^j_A). (\delta_{i\ast}(\tilde{B} - \Delta^j_B)) - D^i_A.\Delta^i_B - \Delta^i_A.D^i_B +  \Delta^i_A.\Delta^i_B$$
By distributivity,  Lemma \ref{pushpullLemma} and the assumption that $\delta^{\ast}_j\delta_{j\ast}\Delta^i_A = \Delta^i_A$ and $\delta^{\ast}_j\delta_{j\ast}\Delta^i_B = \Delta^i_B$  we have,
\begin{align*}
A.B = & \  \delta_i^{\ast}(\delta_{_i\ast}\tilde{A}. \delta_{i\ast}\tilde{B}) - \tilde{A}.\Delta^j_B - \Delta^j_A.\tilde{B} + \Delta^j_A.\Delta^j_B  \\ &-  D^i_A.\Delta^i_B - \Delta^i_A.D^i_B +  \Delta^i_A.\Delta^i_B.
\end{align*}

The last three terms are products in  $D_i \times \R$, and $\tilde{A}.\Delta^j_B,  \Delta^j_A.\tilde{B}$ and $ \Delta^j_A.\Delta^j_B$ are products in $V$. By applying  the contraction $\delta_j$ to calculate these three products we obtain: $$ \Delta^j_A.\Delta^j_B -  \tilde{A}.\Delta^j_B -  \Delta^j_A.\tilde{B} =  \Delta^j_A.\Delta^j_B - D^j_A.\Delta^j_B - \Delta^j_A.D^j_B.$$
Combining this with the equation above and we get, 
\begin{align*}
A.B  = & \  \delta^{\ast}_i( \delta_{i\ast}\tilde{A}. \delta_{i\ast}\tilde{B})  - D^j_A.\Delta^j_B - \Delta^j_A.D^j_B +  \Delta^j_A.\Delta^j_B \\ 
 & - D^i_A.\Delta^i_B - \Delta^i_A.D^i_B +  \Delta^i_A.\Delta^i_B,
\end{align*}
which is symmetric in $i$ and $j$ except for the first term $\delta^{\ast}_i( \delta_{i\ast}\tilde{A}. \delta_{i\ast}\tilde{B}) $ which was already shown to be the same for $i = 1, 2$.  So $A.B$ is independent of the contraction chart chosen. 

Dropping our previous assumption, for any cycle we may still write $A= \tilde{A} -\Delta_A^1 - \Delta_A^2 - \Xi_A$, where $\Xi_A$ is a cycle contained in the kernel of both $\delta_{1\ast}$ and $\delta_{2\ast}$.  
  Letting $A^{\prime} = A + \Xi_A$, and analogously for $B$, and using distributivity with respect to either contraction chart we have 
  \begin{equation}
  A.B = A^{\prime}.B^{\prime} - A^{\prime}.\Xi_B - \Xi_A.B^{\prime}  +  \Xi_A.\Xi_B.
  \end{equation}
As seen above, the product $A^{\prime}.B^{\prime}$ does not depend on the choice of chart $\delta_{i\ast}$.   Moreover since $\Xi_A, \Xi_B$ are in the kernels of both $\delta_{i\ast}$ for both $i=1, 2$, the product $\Xi_A, \Xi_B$   descends to $D_{ij} \times \R^2$ where $D_{ij}$ is the matroid corresponding to $M / \{i, j\}$ where $M$ is the matroid of $V$. This doesn't depend on the order of $i$ and $j$, see Section 3.1 of \cite{Ox}. 
The other two products also descend to  $D_{ij} \times \R^2$ as:
\begin{align*}
A^{\prime}.\Xi_B = (\tilde{A} - \Xi_A ). \Xi_B = (D^i_{\tilde{A}} + D^j_{\tilde{B}} - \Xi_A).\Xi_B \\
 \Xi_A.B^{\prime} = \Xi_A.(\tilde{B} - \Xi_B) = \Xi_A.(D^i_{\tilde{B}} + D^j_{\tilde{B}} - \Xi_B)
 \end{align*}
 which are symmetric in $i$ and $j$.

Now we treat the case of two elementary contractions.  Let $\delta:V \longrightarrow V^{\prime}$ be the composition of two elementary open matroidal contractions. First we set up notation to distinguish between the two orderings, similar to the proof of Proposition \ref{prop:homo}.   We will call  $\delta_i:V \longrightarrow V_i$ and 
$\tilde{\delta}_i: V_j \longrightarrow V^{\prime}$ for $i \neq j$. %Then $\mbox{ker}(\delta_i) = v_i$ and $\mbox{ker}(\tilde{\delta}_i) = \delta_j(v_i)$. 
Let $D_i \subset V_i$ be the divisor associated to $\delta_i$ and suppose $D_i = \mbox{div}_{V_i}(f_i)$. Similarly, $\tilde{D_i} \subset V^{\prime}$ will denote the divisor of $\tilde{\delta}_i$ and $\tilde{f_i}$ its function.  
Keeping the notation from the beginning of the proof for $\Delta_A^i$ and $D_A^i$,  we also set:
\begin{align*}
\tilde{\Delta}^i_A =  & \tilde{\delta}^{\ast}_i \tilde{\delta}_{i\ast}A - A \subset V_j \\
\tilde{D}^i_A =  & \text{div}_{\delta_*A}(\tilde{f}_i) \times \R \subset \tilde{D_i} \times \R
\end{align*}

Applying Definition \ref{stableIntMat} first by contracting with $\delta_i$ and then contracting with $\tilde{\delta}_j$ we obtain:
$$A.B = \delta^{\ast}(\delta_{\ast}A.\delta_{\ast}B)  +  C_{i} + \delta^{\ast}_j\tilde{C}_{j}$$ 
Where $$C_{i}  = \Delta^i_A.\Delta^i_B  - \Delta^i_A.D^i_B - D^i_A. \Delta^i_B$$ with these three products calculated in $D_i \times \R$, and 
$$\tilde{C_{j}} = \tilde{\Delta}^j_{\delta_{i\ast}A}.\tilde{\Delta}^j_{\delta_{i\ast}B} -  \tilde{\Delta}^j_{\delta_{i\ast}A}.\tilde{D}^j_{\delta_{i\ast}B} - \tilde{D}^j_{\delta_{i\ast}A}. \tilde{\Delta}^j_{\delta_{i\ast}B}$$
with each product being calculated in  $\tilde{D}_j \times \R$. 

Again we first assume  that $A = \delta^{\ast}\delta_{\ast}A + \Delta^1_A + \Delta^2_A$. Then we have $\delta^{\ast}_i\tilde{\Delta}^j_{\delta_{i\ast}A} =  \Delta^j_A$. Restricting $\delta_j$ to $D_i \times \R$ we get an elementary open matroidal modification $\delta_j: D_i \times \R \longrightarrow \tilde{D}_i \times \R$. This can be checked on the level of the corresponding matroids. The divisor $\tilde{D}_i$ corresponds to the matroid $M \backslash j /i$ and contracting $D_i$ by $\delta_j$ corresponds to $M /i \backslash j$. By Proposition $3.1.26$ of \cite{Ox} these matroids are equal. Now applying  Lemma \ref{pushpullLemma} for the products in $\tilde{D}_i \times \R$ we have, $C_i = \delta^{\ast}_i\tilde{C}_i$ and   we obtain the same cycle regardless of order. 

The general case follows an argument similar to the general case of two distinct elementary contractions to $\R^k$. We can once again write $A = \delta^{\ast}\delta_{\ast}A - \Delta_A^1 -\Delta_A^2 - \Xi_A$ and similarly for $B$. The rest of the argument follows exactly as above with the products in the end being in $D_{ij} \times \R^2$, where again $D_{ij}$ corresponds to the matroid $M/\{i, j\}$.  
\end{proof}

Now for weakly transverse intersections in a $k$-dimensional matroidal fan $V$ we can make use of the definition of stable intersection in $\R^k$. For each facet $F$ of $V$  we can find a contraction chart $\delta: V \longrightarrow \R^k$ which does not collapse the face $F$. Recall, each facet of $V$ corresponds to a maximal chain in the lattice of flats of the corresponding matroid.
If after deleting an element $i$ from the matroid the chain corresponding to $F$ is still of length $k+1$, the tropical contraction $\delta_i$ of the Bergman fan does not collapse the face $F$. If the chain is of length $k+1$ on $n+1$ elements we can find $n-k$ elements to delete and not collapse $F$.  Using this contraction chart  to calculate the multiplicity we arrive at the following corollary.

\begin{corollary}\label{defTransInt}
Let $V \subset \R^n$ be a matroidal fan and suppose the intersection of the two subcycles $A, B \subset V$ is weakly transverse when restricted to an open facet $F \subset V$ then $A.B \cap F$ corresponds to the stable intersection  of Definition \ref{stabRn}.
\end{corollary}

\begin{proposition}\label{proploc}
For two cycles $A, B$ in a matroidal fan $V \subset \R^n$, the product $A.B$ is supported on $(A \cap B)^{(m)}$, where $m$ is the expected dimension of intersection.
\end{proposition}

\begin{proof}
Once again our proof goes by induction.  Given a facet $F$ of $A.B$, choose a elementary open matroidal  contraction chart $\delta: V \longrightarrow V^{\prime}$ which does not contract the face $E \subset V$ containing $F$.  Again, we may take any chart which does not contract all of the facets adjacent to $E$. 
Let $f$ be the function on $V^{\prime}$ giving the modification $\delta$ and $D$ the corresponding divisor. Then $F$ is contained in $\Gamma_{V^{\prime}}$ the graph of $f$. Let $\Gamma_D \subset \Gamma_{V^{\prime}}$ be the graph of $f$ restricted to $D$. If $\delta(F) \not \subset D$ % \subset \Gamma_{V^{\prime}} \backslash \Gamma_D$ 
then $\delta(F)$ must be a facet of $\delta_{\ast}A.\delta_{\ast}B$. By induction  $\delta(F) \subset (\delta_{\ast}A \cap \delta_{\ast}B)^{(m)}$, %and  since $\delta(F) \not \subset D$
so we must have $F \subset (A \cap B)^{(m)}$. 

If on the other hand $F \subset \Gamma_D$ then $\delta(F)$ is an $m$ dimensional face contained in $\delta_{\ast}A \cap \delta_{\ast}B \cap D$ where $D$ the divisor of the elementary modification $\delta$ and $F$ must be in one of the products $\Delta_A.D_B$, $\Delta_B.D_A$ or $\Delta_A. \Delta_B$ which occur in the fan $D \times \R$. Then assuming the statement holds on $D \times \R$, the facet $F$ must be in one of  $( \Delta_A \cap D_B)^{(m)} \cap \Gamma_D$, 
 $( D_A \cap \Delta_B)^{(m)} \cap \Gamma_D$, or $( \Delta_A \cap \Delta_B)^{(m)} \cap \Gamma_D$.  In any of these three cases $F$ must be a facet of $(\Gamma_{\delta_{\ast}A} \cap \Gamma_{\delta_{\ast}B})^{(m)} \cap \Gamma_D$, and so in $(A \cap B)^{(m)}$. 
\end{proof}

\end{section}

\begin{section}{Two dimensional matroidal fans}\label{sec:surfaces}
%:Examples: surfaces and applications

In this section we consider $V \subset \R^n$ a two dimensional matroidal fan, and $A, B \subset V$  one dimensional fan tropical cycles. This means the vertex of $A$ is the vertex of $V$, and similarly for $B$.  
 Firstly, we simplify the definition of the intersection product  given  in the last section in this case. 

For a two dimensional matroidal fan $V \subset \R^n$ we will consider the \textbf{coarse} subdivision on $V$ described in general by Ardilia and Klivans in \cite{ArdKli}. Suppose $V$ corresponds to a matroid $M$ which is loopless and contains no double points. A \textbf{double point} is an element $i \in E$ such that $r_M(\{i, j\}) = 1$ for some $j \in E$. Recall we defined the \textbf{fine} subdivision on $V$ as the polyhedral complex $B(M)$ described in Section \ref{sec:Berg}. When $V$ is of dimension two and the corresponding matroid satisfies the above assumptions,  the \textbf{coarse} subdivision of $V$ is obtained from $B(M)$ by removing one-dimensional cones corresponding to flats $M$ which are of rank two and size two or of size one and contained in exactly two flats of rank two. See \cite{ArdKli} for more details on the fine and coarse subdivisions of $V$. 

Let 
 $\delta: V \longrightarrow V^{\prime}$, be an elementary open matroidal  modification, $f$ be the associated function on $V^{\prime}$ and $D$ its divisor. %The divisor $D$ is a one dimensional fan, so topologically a vertex of valence $k$.   
 The cycle $\Delta_A$ as defined in the last section is also a fan cycle and thus it is a union of rays.  For any ray $\sigma_i \subset \Delta_A$ contained in the interior of a facet $P_l$ of the coarse subdivision of  $V$ we may write $\sigma_i = <p_iv^l_1 + q_iv^l_2>$, with $p_i, q_i \in \N$ with $(p_i, q_i) =1$, 
where $v^l_1, v^l_2$ are the primitive integer vectors corresponding to flats of $M$ and spanning the one dimensional faces bounding $P_l$. 
Call $A_{\sigma_i}$ the $1$-cycle with three rays each in the directions of $\sigma_i$, $v^l_1$ and $v^l_2$ with weights $ -w_A(\sigma_i)$, $w_A(\sigma_i)p_i,$ and $w_A(\sigma_i)q_i$ respectively. Then the cycle $A_{\sigma_i}$ is contained in the closure of $P_l$. Summing over the facets we get: $$\Delta_A = \sum_{l = 1}^k \sum_{\sigma_i \subset P_l^o} A_{\sigma_i}.$$

Given another fan subcycle $B \subset V $ we have an analogous decomposition $$\Delta_B = \sum_{l = 1}^k \sum_{\tau_j \subset P_l^o} B_{\tau_j}$$  where $\tau_j = <r_j v^l_1 + s_j v^l_2>$ when $\tau_j \subset P_l^o$ and the cycles $ B_{\tau_j} \subset P^o_l$ consist of the three rays  $\tau_j$, $v^l_1$ and $v^l_2$ with weights $ -w_A(\tau_j)$, $w_A(\tau_j)r_j,$ and $w_A(\tau_j)s_j$ respectively.

Since all of $A, B, D$ are fans,  the cycles $D_A$ and $D_B$ (as defined in the last section),  are supported on the one skeleton of $D \times \R$. Take any contraction of $D \times \R$ which preserves the face $F_l$ containing $A_{\sigma_i}$. Then $D_B$ gets sent to an affine line in $\R^2$  with $A_{\sigma_i}$ contained in a halfplane, so  $A_{\sigma_i}.D_B = 0$ similarly with the roles of  $A$ and $B$ interchanged.  Moreover, if $\sigma_i$ and $\tau_j$ are in different facets then  $A_{\sigma_i}.B_{\tau_j} = 0$. 
The multiplicity of the vertex in $A.B$ becomes,  $$m_{A.B}(v) = m_{\delta_{\ast}A. \delta_{\ast}B}(\delta(v)) + m_{\Delta_A.\Delta_B}(v) = m_{\delta_{\ast}A. \delta_{\ast}B}(\delta(v))  + \sum_{l= 1}^k \sum_{ \sigma_i, \tau_j \subset P^o_l}    A_{\sigma_i}.B_{\tau_j}.$$

The intersection of two  cycles $A_{\sigma_i}, B_{\tau_j}$ in a face $F_l$ can be calculated as stable intersection in $\R^2$ by Corollary \ref{defTransInt} . 
If $$\frac{s_j}{r_j} \geq \frac{p_i}{q_i}$$ then we can translate one of the two cycles so that they intersect in exactly one point of multiplicity $-w_A(\sigma_i)w_B(\tau_j) p_ir_j$. Otherwise, $$\frac{p_i}{q_i} > \frac{s_j}{r_j}$$ and we can find a translation so that the two cycles intersect in exactly one point of multiplicity $-w_A(\sigma_i)w_B(\tau_j) s_iq_j$. 
We have just demonstrated the following proposition. 

\begin{proposition}\label{surface}
Let $V \subset \R^n$ be a two dimensional matroidal fan with vertex $v$  and  suppose $A, B \subset V$ are fan cycles and $v \in (A \cap B)^{(0)}$.  Given a elementary contraction $\delta: V \longrightarrow V^{\prime}$, and using the above notation, we have:
$$ m_{A.B}(v) = m_{\delta_{\ast}A.\delta_{\ast}B}(\delta(v)) -  \sum_{l= 1}^k \sum_{ \sigma_i, \tau_j \subset P^o_l} w_A(\sigma_i)w_B(\tau_j) \min \{p_ir_j, q_is_j \}.  $$
\end{proposition}

Using this formula we prove the claim stated at the end of Example \ref{M0n} by calculating the self intersection of $D \subset V$ where $D$ is the divisor of any elementary open matroidal modification $\delta: \mathcal{M}^{trop}_{0,n} \longrightarrow V$. 
The fans resulting from each elementary open matroidal contraction of $\mathcal{M}^{trop}_{0,5}$ are shown in Figure \ref{M05}. The fan $V$ is depicted by the graph in the upper right hand corner and $D \subset V$ is drawn in white. Performing another elementary  contraction $\delta^{\prime}: V \longrightarrow V^{\prime}$ we obtain $$m_{D.D}(v) = m_{\delta^{\prime}_{\ast}D.\delta^{\prime}_{\ast}D}(\delta^{\prime}(v)) - 1.$$ Let $\delta^{\prime \prime}: V^{\prime} \longrightarrow \R^2$ be the composition of the last two contractions, in fact here we have $\delta^{\prime \prime \ast}\delta^{\prime \prime}_{\ast}D = D$, so  by Lemma \ref{pushpullLemma} we have  $m_{D.D}(v) = -1.$  By  part $(5)$ of Proposition \ref{propMatInt} $D.D = \mbox{div}_D(f)$  so by Corollary \ref{regEff} the function giving the elementary open matroidal modification $\delta: \mathcal{M}_{0,5} \longrightarrow V$ cannot be a regular function on $\R^4$.

\comment{
\begin{example}
Denote the Bergman fans of the Fano plane and the anti-Fano by $F$ and $F^-$ respectively.  Both are contained in $\R^6$, and they can both be contracted to  $\mathcal{M}^{trop}_{0,5} \subset \R^5$ via a elementary contraction, see Example 1.5.6 of \cite{Ox}. Let $D_F, D_{F^-} \subset \mathcal{M}^{trop}_{0,5}$ denote the divisors of the contractions of $F$ and $F^-$ respectively. The rank two matroids corresponding to $D_F$ and $D_{F^-}$  are a  $3$-valent tropical line and $D^-$ a $4$-valent tropical line respectively. Moreover if $\delta: \mathcal{M}^{trop}_{0,n} \longrightarrow V$ is from the example above, $$\delta_{\ast}D_F = \delta_{\ast}D_{F^-} = D$$ where $D$ is also the divisor from the example above. Moreover it can be checked that $\delta^{\ast}\delta_{\ast}D_{F^-}  = D_{F^-} $, then by Lemma \ref{pushpullLemma} and Proposition \ref{surface}  we have $$D_F.D_{F^-} = -1\cdot v$$ where $v$ is the vertex of $\mathcal{M}^{trop}_{0,5}$. The map $\delta$ contracts a facet of $\mathcal{M}^{trop}_{0,5}$ containing a ray from $D_F$ with slope $(1,1)$ relative to the two one-cones generating this face. By Proposition \ref{surface}, we obtain: $$D_F.D_F = -2\cdot v.$$
\end{example}}

In general, a one dimensional matroidal fan cycle in a two dimensional matroidal fan $V$ corresponds to a matroid quotient of $V$. The next proposition describes the intersection multiplicity of two matroidal one cycles in $V$ in terms of the lattices of flats of the corresponding matroids. 
%Our product restricted to matroidal sub-fans of $V$ gives an intersection product on quotients of matroids. The above example shows the product may yield non-effective cycles.

\begin{theorem}\label{thm:linesflats}
Let $V_M \subset \R^n$ be a two dimensional matroidal fan corresponding to the matroid $M$ and $L_1, L_2 \subset V_M$ be two one dimensional tropical cycles corresponding to matroids $M_1$ and $M_2$ respectively. Then their  intersection multiplicity at the vertex of $V_M$ is
$$m_{L_1.L_2}(v) = 1 - |\{F \ | \ r_M(F) = 2, \ F \in \Lambda(M_1) \cap   \Lambda(M_2)  \}|,$$
where $\Lambda(M_k)$ is the lattice of flats of $M_k$ and $r_M$ is the rank function on $M$. 
\end{theorem}

\begin{proof}
To start, note that by a verification of the possible lines in $\R^2$ the theorem holds for $V_M = \R^2$. 
Given a  two dimensional  matroidal fan $V_M \subset \R^n$,  let $\delta: V_M \longrightarrow V_{M \backslash i}$ denote a principal open matroidal modification. We may assume by induction that the given formula for the intersection multiplicity holds for $\delta_*L_1 = L_1^{\prime}$ and $\delta_*L_2 = L_2^{\prime}$ in $V_{M \backslash i}$. Now $L_k^{\prime}$ corresponds to the matroid $M_k \backslash i$ for $k = 1, 2$. 
Letting $\delta(v) = v^{\prime}$ we obtain:
\begin{equation}\label{flats one}
m_{L_1^{\prime}.L_2^{\prime}}(v^{\prime}) = 1 - |\{ F \ | \ r_{M\backslash i}(F) = 2, \ F \in \Lambda(M_1\backslash i) \cap   \Lambda(M_2 \backslash i)  \}|,
\end{equation}
and $m_{\delta^*\delta_*L_1.\delta^*\delta_*L_2}(v) =  m_{L_1^{\prime}.L_2^{\prime}}(v^{\prime}).$

%Since,  $M_1$ and $M_2$ are of rank two the rays of $L_k$ are in correspondence with the flats of $M_k$ for $k =1,2$. 
A ray $\sigma_F$ of $L_k$  corresponding to a flat $F \in \Lambda(M_k) \backslash \{ \emptyset, E\}$ is contained in the interior of a facet of the undergraph $\mathcal{U}_f(V_{M\backslash i})$ considered with  the coarse subdivision if and only if the corresponding flat $F$ is of rank  two in $M$  and contains $i$. 
The weights of all edges of $L_1, L_2$ are equal to one, so in this situation the simplification given in Proposition \ref{surface} yields, 
\begin{equation}\label{flats two}
m_{L_1.L_2}(v) = m_{L_1^{\prime}.L_2^{\prime}}(v^{\prime})
 - | \{ F \ | \
r_M(F) = 2, \
F \in \Lambda(M_1) \cap   \Lambda(M_2) , \ 
i \in F
\}.\end{equation}
The combination of Equations \ref{flats one} and \ref{flats two} along with the description of the flats of $M \backslash i$ and $M / i$ given in the proof of Proposition \ref{matroidMod} proves the intersection multiplicity for $L_1$ and $L_2$ in $V_M$. 
% \emptyset \neq \mathcal{F} \backslash e \in \Lambda_{M|e}\}|  $$
% \sum_{l= 1}^k \sum_{ \sigma_i, \tau_j \subset F^o_l}  \min \{p_ir_j, q_is_j \}. $$
\end{proof}

It is possible generalise the above proposition to a matroidal fan $V_M$ of any dimension and describe combinatorially the product of two fans corresponding to matroidal quotients $V_1, V_2 \subset V_M$. This product on matroidal quotients generalises the standard matroid intersection, which is shown by Speyer to correspond to tropical stable intersection under certain conditions, \cite{Speyer}. 

\begin{example}
Denote the Bergman fans of the Fano plane and the anti-Fano by $V_F$ and $V_{F^-}$ respectively.  Both are contained in $\R^6$, and they can both be contracted to  $\mathcal{M}^{trop}_{0,5} \subset \R^5$ via an elementary open matroidal  contraction, see Example 1.5.6 of \cite{Ox}. Let $D_F, D_{F^-} \subset \mathcal{M}^{trop}_{0,5}$ denote the divisors of the contractions of $F$ and $F^-$ respectively. Then  $D_F$ is a tropical one cycle with three rays   and $D_{F^-}$  is a tropical one cycle with four  rays.  The cycles $D_F$ and $D_{F^-}$ have two rays in common. Moreover,  these common rays correspond to flats of rank two in the matroid corresponding to $\mathcal{M}_{0,5}^{trop}$, therefore by Theorem \ref{thm:linesflats} we obtain, 
$$m_{D_F. D_{F^-}}(v) = -1.$$

\comment{
In fact, these correspond to flats of $\mathcal{M}^{trop}_{0,5}$
Moreover if $\delta: \mathcal{M}^{trop}_{0,n} \longrightarrow V$ is from the example above, $$\delta_{\ast}D_F = \delta_{\ast}D_{F^-} = D$$ where $D$ is also the divisor from the example above. Moreover it can be checked that $\delta^{\ast}\delta_{\ast}D_{F^-}  = D_{F^-} $, then by Lemma \ref{pushpullLemma} and Proposition \ref{surface}  we have $$D_F.D_{F^-} = -1\cdot v$$ where $v$ is the vertex of $\mathcal{M}^{trop}_{0,5}$. The map $\delta$ contracts a facet of $\mathcal{M}^{trop}_{0,5}$ containing a ray from $D_F$ with slope $(1,1)$ relative to the two one-cones generating this face. By Proposition \ref{surface}, we obtain: $$D_F.D_F = -2\cdot v.$$}
\end{example}

A negative intersection multiplicity of two  effective subcycles is under some circumstances  an indication that these cycles cannot both arise as tropicalisations of classical varieties. 
The following theorem makes this precise and is due to an observation of  E. Brugall\'e. Here, let $\textbf{K}$ be the field of  Puiseux series with coefficients in an algebraically closed field $\textbf{k}$, and let $\text{Trop}(\textbf{V}) \subset \R^n$ denote the tropicalisation of a subvariety $\textbf{V} \subset (\textbf{K}^*)^n$ from \cite{St7}.  We say that a subvariety $\textbf{V} \subset (\textbf{K}^*)^n$ is a plane if it is two dimensional and  defined by a system of linear equations. 
Let $\tilde{\textbf{V}} \subset (\textbf{K}^*)^n$ be a plane and 
$\Delta:  (\textbf{K}^*)^n \longrightarrow  (\textbf{K}^*)^{n-1}$ be the projection by forgetting a coordinate direction, then $\Delta(\tilde{\textbf{V}}) = \textbf{V}$ is also a plane.  Let $\text{Trop}(\tilde{\textbf{V}}) = \tilde{V}$ and $\text{Trop}(\textbf{V}) = V$, then there is a tropical modification $\delta: \tilde{V} \longrightarrow V$. Denote its corresponding divisor $D \subset V$. Using this notation we have the following theorem.

\begin{theorem}\label{realisable}  
Let $\tilde{\textbf{V}} \subset (\textbf{K}^*)^n$ be a plane and let $\textbf{V} \subset  (\textbf{K}^*)^{n-1}$ be the projection of $\tilde{\textbf{V}}$ along one of the coordinate directions. Suppose $\text{Trop}(\tilde{\textbf{V}} ) = \tilde{V} \subset \R^{n}$ and $\text{Trop}(\textbf{V}) = V \subset \R^{n-1}$ and let $\delta: \tilde{V} \longrightarrow V$ be an open elementary tropical modification with divisor $D \subset V$. 
%Let $\textbf{K}$ be a field and suppose $\tilde{V} \subset \R^n$ is a matroidal fan realised over $\textbf{K}$ by some plane $\tilde{\textbf{V}} \subset (\textbf{K}^*)^n$. Suppose we have a modification $ \delta: \tilde{V} \longrightarrow V$ with divisor $D$. 
 Given a tropical curve $C \subset V$ such that $D \not \subset C$, if there exists a bounded connected subset $Q \subset C\cap D$ such that 
 $$\sum_{p \in Q \cap (D \cap C)^{(0)} } m_p(D.C) < 0$$ then there is no algebraic curve $\textbf{C} \subset \textbf{V}$ such that $\text{Trop}(\textbf{C}) = C$. 
 %  $C \subset V$ is not realisable by a pair $\textbf{C} \subset \textbf{V}$ where $\textbf{C}$ is an algebraic curve and $\textbf{V}$ is a plane over $\textbf{K}$.
\end{theorem}

\begin{proof}
%By assumption we have  $\mbox{val}(\tilde{\textbf{V}}) = \tilde{V}$. 
%Moreover, there exists a projection  $\Delta: (K^*)^n \longrightarrow (K^*)^{n-1}$ for which $$\overline{\Delta(\tilde{\textbf{V}})} = \textbf{V}$$ and a $k-1$ plane $\textbf{D} \subset \textbf{V}\subset (K^*)^{n-1}$ such that 
%\begin{align*}
%\mbox{val}(\textbf{V}) = & \ V \\
%\mbox{val}(\textbf{D}) = & \  D.
%\end{align*}
The plane $\tilde{\textbf{V}} \subset (\textbf{K}^*)^n$ is obtained by taking the graph of a linear function $\textbf{f}$ on  $\textbf{V} \backslash \textbf{D}$ where $\textbf{D}$ is the divisor of $\textbf{f}$ on $\textbf{V}$. 
Suppose there exists a $\textbf{C} \subset \textbf{V}$ such that $\text{Trop}(\textbf{C}) = C$. 
Then $\textbf{f}$ also gives  an embedding $\textbf{C} \backslash \textbf{C}\cap \textbf{D}  \longrightarrow \tilde{\textbf{V}} \subset (\textbf{K}^*)^n$ given by taking the graph of $\textbf{f}$ restricted to $\textbf{C} \backslash \textbf{C} \cap \textbf{D}$. Let $\tilde{\textbf{C}}$ denote the image of this embedding and let  $\tilde{C} = \text{Trop}(\tilde{\textbf{C}})$. By Theorem 3.3.4 of \cite{St7} there exist positive weights on the facets of $\tilde{C}$ making it a balanced cycle. 
However, the pullback $\delta^*C$ is not-effective,  in particular for each point $p \in Q$ with $m_p(D.C) \neq 0$ there is a corresponding half-ray in $\delta^*C$ in the direction $-e_n$ of weight $m_p(D.C)$, whose image under $\delta$ is the point $p$. The cycles $\tilde{C}$ and $\delta^*C$ agree as weighted complexes outside of $\delta^{-1}(D \cap C)$. Moreover, the difference  $\tilde{C} - \delta^*C$ is a cycle and  has a connected component contained in $\delta^{-1}(Q)$. Since $Q$ is bounded, all of the unbounded rays of $\tilde{C} - \delta^*C$ in $\delta^{-1}(Q)$ must have primitive integer direction $-e_n$. The recession fan of $\tilde{C} - \delta^*C$ is also balanced, meaning the sum of the weights of the unbounded edges of $\tilde{C} - \delta^*C$ must also be equal zero. 
However, the sum of the weights of the unbounded edges of $\delta^*C$ is given by  $$\sum_{p \in Q \cap (D \cap C)^{(0)} } m_p(D.C) < 0, $$ therefore the cycle $\tilde{C}$ may not be effective. This contradiction proves the theorem. 
\end{proof}

\begin{proposition}
Recall the curve $B$ from Example \ref{negCurve}. For $d \geq 3$, $B \subset V$ is not realisable over any field.  
\end{proposition}

\begin{proof}
By the above theorem it suffices to show that the matroid corresponding to $\tilde{V}$ which is the fan  obtained by the modification $\delta: \tilde{V} \longrightarrow V$ along the matroidal divisor $A$  is a  regular matroid, i.e. realisable over every field. For a matroid of this rank on only five elements we must only check that it has no minors corresponding to the four point line, see Theorem 6.6.4 of \cite{Ox}. Tropically this means that the divisor of any contraction cannot be the four valent tropical line $L \subset \R^3$. Verifying the five possible contractions and we see that it holds. 
\end{proof}

This is a light version of a much stronger result which should hold not just in open Bergman fans but in their compactifications as well and in non-singular tropical varieties. 

Unfortunately, there are some tropical $1$-cycles which are not realisable which pass this \textit{intersection test}. For example, Vigeland's $1$-parameter family of lines on a degree $d \geq 3$ surface, see Theorem 9.3 of \cite{VigInfinite}.  

\end{section}

\bibliography{bibarticle}{}
\bibliographystyle{plain}

\end{document}